\pdfoutput=1
\documentclass[11pt]{article}

\usepackage{amssymb,amsmath,amsthm,mathtools}

\usepackage{graphicx,color}
\graphicspath{{figures/}{figures/simulations/}}

\usepackage[unicode,colorlinks=true,citecolor=blue,linkcolor=blue]{hyperref}

\usepackage[twoside,centering,a4paper,bottom=2.7cm,textwidth=15.24cm]{geometry}

\def\inserttitle{}
\def\insertshorttitle{}
\def\insertauthor{}
\def\insertshortauthor{}

\makeatletter
\renewcommand{\title}[2][]{%
  \def\inserttitle{#2}%
  \def\@title{#2}%
  \ifx&#1&%
  \def\insertshorttitle{#2}%
  \else%
  \def\insertshorttitle{#1}%
  \fi%
  \markboth{\hfill\scshape\insertshortauthor\hfill}%
  {\hfill\scshape\insertshorttitle\hfill}
}
\renewcommand{\author}[2][]{%
  \def\insertauthor{#2}%
  \def\@author{#2}%
  \ifx&#1&%
  \def\insertshortauthor{#2}%
  \else%
  \def\insertshortauthor{#1}%
  \fi%
  \markboth{\hfill\scshape\insertshortauthor\hfill}%
  {\hfill\scshape\insertshorttitle\hfill}
}
\makeatother

\renewenvironment{abstract}{\footnotesize\begin{quote}\textbf{Abstract.}~}%
  {\end{quote}}

\newcommand{\keywords}[1]{\par\emph{Keywords}:~#1.\par}
\newcommand{\msc}[1]{\par\emph{MSC(2010)}:~#1.\par}

\theoremstyle{plain}
\newtheorem{theorem}             {Theorem}    [section]
\newtheorem{lemma}      [theorem]{Lemma}

\newtheorem{corollary}  [theorem]{Corollary}

\theoremstyle{definition}

\theoremstyle{remark}
\newtheorem*{remark}             {Remark}

\newcommand{\bbN}{\mathbb{N}}

\newcommand{\bbR}{\mathbb{R}}

\newcommand{\calA}{\mathcal{A}}
\newcommand{\calC}{\mathcal{C}}
\newcommand{\calD}{\mathcal{D}}
\newcommand{\calE}{\mathcal{E}}
\newcommand{\calF}{\mathcal{F}}
\newcommand{\calG}{\mathcal{G}}
\newcommand{\calH}{\mathcal{H}}
\newcommand{\calI}{\mathcal{I}}
\newcommand{\calL}{\mathcal{L}}

\newcommand{\calP}{\mathcal{P}}
\newcommand{\calQ}{\mathcal{Q}}
\newcommand{\calT}{\mathcal{T}}
\newcommand{\calV}{\mathcal{V}}

\renewcommand{\phi}{\varphi}
\renewcommand{\epsilon}{\varepsilon}
\newcommand{\eps}{\varepsilon}

\DeclareMathOperator*{\osc}{osc}

\DeclareMathOperator{\sign}{sign}
\DeclareMathOperator{\trace}{trace}

\DeclareMathOperator{\diag}{diag}
\renewcommand{\div}{\mathrm{div}}
\DeclareMathOperator{\spt}{spt}
\DeclareMathOperator{\interior}{int}

\newcommand{\snabla}[2][\empty]{\ensuremath{%
    \ifx\empty#1%
    \nabla #2%
    \else%
    \nabla_{\!\!#1} #2%
    \fi%
  }%
}

\newcommand{\sdiv}[2][\empty]{\ensuremath{%
    \ifx\empty#1%
    \div #2%
    \else%
    \div_{\!#1} #2%
    \fi%
  }%
}

\newcommand{\slaplace}[2][\empty]{\ensuremath{%
    \ifx\empty#1%
    \Delta #2%
    \else%
    \Delta_{#1} #2%
    \fi%
  }%
}

\newcommand{\wto}{\rightharpoonup}
\newcommand{\wsto}{\stackrel{*}{\rightharpoonup}}

\newcommand{\loc}{\mathrm{loc}}

\newcommand{\sm}{\setminus}

\newcommand{\set}[2][\empty]{\ensuremath{%
    \left\{%
      \ifx\empty#1%
      \relax%
      \else%
      #1:%
      \fi%
      #2%
    \right\}%
  }%
}

\numberwithin{figure}{section}
\numberwithin{table}{section}
\numberwithin{equation}{section}

\title{Kinks in two-phase lipid bilayer membranes}

\author[M. Helmers]{%
  {\scshape Michael Helmers}\\%
  {\footnotesize Institute for Applied Mathematics, University of Bonn}\\[-.6ex]
  {\footnotesize Endenicher Allee 60, 53115 Bonn, Germany}\\[-.6ex]%
  {\footnotesize Email: {\tt helmers@iam.uni-bonn.de}}%
}

\date{}

\begin{document}

\maketitle
\thispagestyle{empty}


\begin{abstract}
  Common models for two-phase lipid bilayer membranes are based on an
  energy that consists of an elastic term for each lipid phase and a
  line energy at interfaces. Although such an energy controls only the
  length of interfaces, the membrane surface is usually assumed to be
  at least $C^1$ across phase boundaries.
  We consider the spontaneous curvature model for closed rotationally
  symmetric two-phase membranes without excluding tangent
  discontinuities at interfaces a~priorily. We introduce a family of
  energies for smooth surfaces and phase fields for the lipid phases
  and derive a sharp interface limit that coincides with the
  $\Gamma$-limit on all reasonable membranes and extends the classical
  model by assigning a bending energy also to tangent discontinuities.
  The theoretical result is illustrated by numerical examples.
  

  \msc{%
    49J45, 
    82B26, 
    49Q10, 
    92C10} 
\end{abstract}



\section{Introduction}

Lipid bilayer membranes are the building block of numerous biological
systems and appear in a rich variety of shapes. In particular,
membranes consisting of two or more lipid phases display a complex
morphology, which is affected by their elastic properties as well as
phase separation \cite{SeBeLi91,DoKaNoSpSa93,BaHeWe03}.

A well-established model for the shape of two-phase vesicles is the
spontaneous curvature model, where equilibrium shapes are described as
surfaces minimising the energy
\begin{equation}
  \label{eq:intro:energy}
  \sum_{j=\pm}
  \int_{M^j} k^j ( H-H_s^j )^2 + k_G^j K \,d\mu
  +
  \sigma \calH^1(\partial M^+)
\end{equation}
among all closed surfaces $M = M^+ \cup M^- \cup \partial M^+$, $M^+
\cap M^- = \emptyset$ with prescribed areas for $M^\pm$ (and
prescribed enclosed volume) \cite{Canham70, Evans74, Helfrich73,
  JuLi96, SeLi95}.
Here $H$ and $K$ are the mean and the Gauss curvature of the membrane
surface $M$, and $\mu$ is its area measure. The bending rigidities
$k^\pm > 0$ and the Gauss rigidities $k_G^\pm$ are elastic material
parameters, and $H_s^\pm$, the so-called spontaneous or preferred
curvatures, are supposed to reflect an asymmetry in the membrane.
In the simplest case, the rigidities and spontaneous curvatures are
constant within each lipid phase but different between the two phases.

Apart from the length $\calH^1(\partial M^+)$ of phase boundaries
multiplied by a constant line tension $\sigma$,
\eqref{eq:intro:energy} does not control the membrane surface at the
interface $\partial M^+ = \partial M^-$; studies of two-phase
membranes, however, commonly include an a~priori smoothness
assumption. J{\"u}licher and Lipowsky \cite{JuLi96} consider the
Euler-Lagrange equations of \eqref{eq:intro:energy} for axially
symmetric membranes that have exactly one interface between the two
lipid phases and are $C^1$ across this interface. Du, Wang
\cite{DuWa08} and Lowengrub, R{\"a}tz, Voigt \cite{LoRaVo09} perform
numerical simulations using a phase field for both the membrane and
the lipid phases; Elliot and Stinner \cite{ElSt10_Modeling,ElSt10}
consider a surface phase field model. Convergence to the sharp
interface limit in these approaches is obtained by asymptotic
expansion and under strong smoothness assumptions on the limit
surface; in particular, the membrane is again assumed to be at least
$C^1$ across interfaces.
In \cite{Helmers12} we prove that for rotationally symmetric membranes
this regularity need not be assumed, but is included in the
$\Gamma$-limit of an appropriate surface phase field approximation.

Mathematically, however, the natural setting for the energy
\eqref{eq:intro:energy} does not contain $C^1$ regularity across
interfaces. Moreover, the numerical simulations in
\cite{DuWa08,ElSt10_Modeling} show that equilibrium shapes of models
including $C^1$ regularity have rather ample neck regions, see also
Figure~\ref{fig:kink} on the left, while shapes observed in
experiments do not \cite{BaHeWe03}.
In this paper we study a diffuse interface approximation for the lipid
phases of rotationally symmetric membranes whose sharp interface limit
allows tangent discontinuities or kinks at interfaces, thus
infinitesimally small neck regions. More precisely, for a closed
surface $M_\gamma$ obtained by rotating a curve $\gamma$ about the
$x$-axis and an associated rotationally symmetric phase field $u
\colon M_\gamma \to \bbR$, we consider an approximate energy of the
form
\begin{equation}
  \label{eq:intro:approx-energy}
  \int_{M_\gamma}
  u^2 k(u) (H-H_s(u))^2 + u^2 k_G(u) K \,d\mu
  +
  \int_{M_\gamma}
  \eps |\snabla[M_\gamma] u|^2 + \frac{1}{\eps} W(u) \,d\mu
  +
  \int_{M_\gamma}
  \eps |B|^2 \,d\mu.
\end{equation}
Apart from the surface setting, the second integral in
\eqref{eq:intro:approx-energy} is the usual Modica-Mortola
approximation of the interface energy. With a standard double well
potential such as $W(u)=(1-u^2)^2$, the phase field is forced to
$\pm1$ as $\eps\to0$ and the first integral in
\eqref{eq:intro:approx-energy} resembles the curvature energy in
\eqref{eq:intro:energy}, provided that $k$, $k_G$, and $H_s$ are
extensions of $k^\pm = k(\pm1)$, $k_G^\pm = k_G(\pm1)$, and $H_s^\pm =
H_s(\pm1)$.
The third integral, where $B$ denotes the second fundamental form of
$M_\gamma$, is on the one hand necessary for compactness of
energy-bounded sequences as $\eps\to0$. On the other hand, it assigns a
curvature to kinks in the limit by penalising their size, thereby
providing a meaningful extension of \eqref{eq:intro:energy}.

The bending parameters and their extensions play a crucial role in our
approximation. For biological membranes it is well-known that
$k^\pm>0$, but measurements of the Gauss rigidities are
scarce. Available data suggest that $k_G^\pm < 0$ and that the
inequalities $k^\pm > -k_G^\pm/2$ hold at least for some monolayers
\cite{SiKo04,TeKhSe98}.
It turns out that we need similar conditions for the extensions: we
let $k,k_G \colon \bbR \to \bbR$ be continuous and bounded functions
such that $k(\pm1)=k^\pm$, $k_G(\pm)=k_G^\pm$,
\begin{equation}
  \label{eq:intro:param_restrictions}
  \inf_\bbR k > 0,
  \qquad
  \sup_\bbR k_G < 0,
  \qquad\text{and}\qquad
  \inf_\bbR \left(k+k_G/2\right) > 0.
\end{equation}
The preferred curvature extension $H_s \colon \bbR \to \bbR$ can be
any continuous bounded function such that $H_s(\pm1)=H_s^\pm$.
Under these conditions there is a $\delta>0$ such that
$(1-\delta)k>-k_G/2$, and Young's inequality together with $|B|^2 =
H^2 - 2K$ implies
\begin{equation}
  \begin{aligned}
    \label{eq:intro:sec-fform-bound}
    u^2 k &(H-H_s)^2 + u^2 k_G K
    \\
    &=
    - u^2 \frac{k_G}{2} |B|^2
    + u^2 \left( k+\frac{k_G}{2} \right) H^2
    + u^2 k H_s^2 - 2 u^2 k H H_s
    \\
    &\geq
    - u^2 \frac{k_G}{2} |B|^2
    + u^2 \left( (1-\delta)k + \frac{k_G}{2} \right) H^2
    - \frac{1-\delta}{\delta} \| k H_s^2 \|_\infty \| u \|_\infty^2
    \\
    &\geq
    - C \|u\|_\infty^2,
  \end{aligned}
\end{equation}
where $C>0$ depends only on $k$, $k_G$, and $H_s$.
By \eqref{eq:intro:sec-fform-bound} the energy
\eqref{eq:intro:approx-energy} provides weighted $L^2$-bounds for $B$
and $H$, which are used to establish equi-coercivity and a lower bound
inequality.
These bounds degenerate for $u \approx 0$, and similarly to the
Ambrosio-Tortorelli approximation for free discontinuity problems
\cite{AmTo90}, this degeneracy allows curvatures to become large and
yield kinks in the limit.
Interestingly, conditions such as \eqref{eq:intro:param_restrictions}
also appear in \cite{BeMu10}, where the authors obtain a partial
$\Gamma$-convergence result for a diffuse interface approximation of
the membrane surface of single-phase vesicles.

Under the above restrictions on the parameters, we prove that a
limit of \eqref{eq:intro:approx-energy} is given by
\begin{multline}
  \label{eq:intro:limit-energy}
  \int_{M_\gamma(\set{y>0} \sm S)}
  k(u) (H-H_s(u))^2 + k_G(u) K \,d\mu
  \\
  +
  2\pi \sum_{s \in S} (\sigma + \hat\sigma |[\gamma'](s)|) y(s)
  +
  2\pi \hat\sigma \calL_\gamma(\set{y=0})
\end{multline}
for membranes $(\gamma,u)$, $\gamma=(x,y)$ with membrane surface
$M_\gamma$ and lipid phases $u \in \set{\pm1}$. Here $S$ denotes the
set of interfaces, that is, the set of jumps of $u$, and of tangent
discontinuities of $\gamma$. An interface is penalised essentially by
its length $2\pi y(s) = \calH^1(M_\gamma(\set{s}))$, while a kink
carries an additional ``bending'' energy $2\pi|[\gamma'](s)|y(s)$,
where $|[\gamma'](s)|$ is the modulus of the angle enclosed by the two
one-sided tangent vectors at $s$ modulo $2\pi$. The constants $\sigma$
and $\hat\sigma$ are given by
\begin{equation}
  \label{eq:intro:sigma-sigma-hat}
  \sigma = \int_{-1}^1 2 \sqrt{W(u)} \,d u
  \qquad\text{and}\qquad
  \hat\sigma = 2\sqrt{W(0)}.
\end{equation}
The set $M_\gamma(\set{y>0} \sm S)$ is the part of $M_\gamma$ that is
obtained by rotating the restriction of $\gamma$ to $\set{y>0} \sm S$,
and $\calL_\gamma(\set{y=0})$ denotes the length of the segment
$\gamma(\set{y=0})$. Here and in the following $\set{y=0}$ is the set
where $\gamma$ lies on the axis of revolution and $\set{y>0}$ the set
where it does not.
A limit membrane $(\gamma,u)$ for which \eqref{eq:intro:limit-energy}
is finite may touch the axis of revolution not only at the end points
of $\gamma$, but also in regions in the interior, see
Figure~\ref{fig:model:curve-in-d} for a non membrane-like example. Our
limit equals the $\Gamma$-limit of $\calF_\eps$ on ``good'' membranes
where $|\gamma'|=\text{const}$, and exactly regions on the axis of
revolution prevent us from obtaining full $\Gamma$-convergence.
However, a ``good'' limit is for instance a membrane that touches the
axis of revolution only at countably many points, hence our result
covers reasonable biological membranes.

The paper is organised as follows. Section \ref{sec:surf} is a brief
recapitulation of surfaces of revolution. In Section \ref{sec:model}
we state our approximate setting, the limit, and the convergence
result; we also present some numerical examples and compare our model
to one without kinks. The proof of the convergence theorem is
presented in Section \ref{sec:proof}, and in Section
\ref{sec:some-generalisations} we consider some generalisations of our
result, including a brief discussion of the full $\Gamma$-limit.


\section{Surfaces of revolution}
\label{sec:surf}

Let $I \subset \bbR$ be an open and bounded interval and $\gamma=(x,y)
\colon I \to \bbR \times \bbR_{\geq0}$ a Lipschitz parametrised curve
in the upper half of the $x y$-plane.
We denote by $M_\gamma$ the surface in $\bbR^3$ obtained by rotating
$\gamma$ about the $x$-axis, that is,  $M_\gamma$ is the image of
$\overline{I} \times [0,2\pi)$ under the Lipschitz continuous map
\begin{equation*}
  \Phi \colon
  (t,\theta) \mapsto (x(t), y(t) \cos\theta, y(t) \sin\theta);
\end{equation*}
$\gamma$ is called generating curve of
$M_\gamma$. 
Since $\gamma$ and $\Phi$ are Lipschitz, they are weakly and almost
everywhere differentiable with bounded derivatives. The length of
$\gamma$, the area measure of $M_\gamma$, and the area of $M_\gamma$
are well-defined and given by
\begin{equation*}
  \calL_{\gamma} =  \int_{I} |\gamma'(t)| \,d t,
  \qquad
  d\mu = |\gamma'| y \,d t \,d\theta,
  \qquad\text{and}\qquad
  \calA_\gamma = 2\pi \int_{I} |\gamma'| y \,d t,
\end{equation*}
respectively. 
If $J$ is a measurable subset of $I$, we denote by $M_\gamma(J)$ the
part of $M_\gamma$ that is obtained by rotating the curve segment
$\gamma(J)$ and write $\calL_\gamma(J)$ and $\calA_\gamma(J)$ for the
corresponding length and area.

After removing at most countably many constancy intervals, pulling holes
together and reparametrising, we may assume that $\gamma$ is
parametrised with constant speed $|\gamma'| \equiv \calL_\gamma/|I| =:
q_\gamma>0$ almost everywhere in $I$~\cite[Lemma 5.23]{BuGiHi98}.
Then the tangent space $\calT_{(t,\theta)} M_\gamma$, which exists
for almost every $(t,\theta) \in I \times [0,2\pi)$, is spanned by the
orthonormal vectors
\begin{equation*}
  \xi_1
  =
  \frac{\partial_t \Phi}{|\partial_t \Phi|}
  =
  \frac{1}{|\gamma'|} \left( x', y' \cos\theta, y' \sin\theta \right)
  \qquad\text{and}\qquad
  \xi_2
  =
  \frac{\partial_\theta \Phi}{|\partial_\theta \Phi|}
  =
  \left( 0, -\sin\theta, \cos\theta \right),
\end{equation*}
and a unit normal is given by
\begin{equation}
  \label{eq:surf:unit-normal}
  \nu
  =
  \frac{\partial_t \Phi \wedge \partial_\theta \Phi}
  {|\partial_t \Phi \wedge \partial_\theta \Phi|}
  =
  \frac{1}{|\gamma'|} \left( -y', x' \cos\theta, x' \sin\theta \right).
\end{equation}
Since $M_\gamma$ is not necessarily embedded, tangent space, normal,
and the quantities defined below are associated to the parameter
$(t,\theta)$ and not to the point $\Phi(t,\theta)$ on the surface
$M_\gamma$.

We consider a function $f \colon M_\gamma \to \bbR^k$ to be a function
$F \colon \overline{I} \times [0,2\pi) \to \bbR^k$ of the parameters;
on embedded parts of $M_\gamma$ this amounts to $f(\Phi(t,\theta)) =
F(t,\theta)$.
Given a tangent vector $\xi$ at $(t_0,\theta_0) \in I \times
(0,2\pi)$, the directional derivative of $f$ in direction $\xi$ is
defined as
\begin{equation*}
  D_\xi f(t_0,\theta_0)
  = \left. \frac{d}{d s} F(\eta(s)) \right|_{s=0},
\end{equation*}
where $\eta \colon (-\delta,\delta) \to I \times [0,2\pi)$ is a
$C^1$-curve satisfying $\eta(0) = (t_0,\theta_0)$ and $ \frac{d}{d s}
\Phi(\eta(s)) \big|_{s=0} = \xi$.
The tangential gradient of $f \colon M_\gamma \to \bbR$ is the vector
\begin{equation*}
  \snabla[M_\gamma]{f} = (D_{\xi_1} f) \xi_1 + (D_{\xi_2} f) \xi_2,
\end{equation*}
and since $D_{\xi_1} f = |\gamma'|^{-1} \partial_t F$, we obtain for a
rotationally symmetric $f$, which does not depend on $\theta$, that
\begin{equation*}
  \snabla[M_\gamma]{f}(t,\theta)
  = \frac{F'(t)}{|\gamma'(t)|} \xi_1(t,\theta)
  \qquad \text{and} \qquad
  |\snabla[M_\gamma]{f}(t,\theta)| = \frac{|F'(t)|}{|\gamma'(t)|},
\end{equation*}
where $| \cdot |$ is the Euclidean norm in $\bbR^3$.


For the rest of this section let $\gamma \in W^{2,1}_{\loc}(I;
\bbR^2)$ be twice weakly differentiable, thus twice differentiable
almost everywhere, and assume that $y>0$ in $I$. Then the normal $\nu$
is weakly differentiable, thus the shape operator $S \colon
\calT_{(t_0,\theta_0)}M \to \calT_{(t_0,\theta_0)}M$, $\zeta \mapsto
D_\zeta \nu$ and the second fundamental form $B \colon
\calT_{(t_0,\theta_0)}M \times \calT_{(t_0,\theta_0)}M \to \bbR$,
$(\zeta,\xi) \mapsto \xi \cdot S\zeta$ are well-defined for almost
every $(t_0,\theta_0)$. The matrix representation with respect to the
basis $\set{\xi_1,\xi_2}$ of both is $\diag(\kappa_1,\kappa_2)$ where
\begin{equation}
  \label{eq:surf:curvatures}
  \kappa_1
  =
  \frac{-y'' x' + y' x''}{|\gamma'|^3}
  \quad\text{and}\quad
  \kappa_2
  =
  \frac{x'}{y|\gamma'|}.
\end{equation}
The eigenvalues $\kappa_1, \kappa_2$ of $S$ are the principal
curvatures of $M_\gamma$, and $\kappa_1$ is just the signed curvature
of $\gamma$ with respect to the normal $\frac{1}{|\gamma'|}(y',-x')$.
The mean curvature $H$ and the Gauss curvature $K$ of $M_\gamma$ are
\begin{equation*}
  H = \trace S = \kappa_1 + \kappa_2
  \qquad\text{and}\qquad
  K = \det S = \kappa_1 \kappa_2.
\end{equation*}
By $|B|^2 = |S|^2 = \kappa_1^2 + \kappa_2^2$ we denote the squared
Frobenius norm of $S$ and $B$.
The signs of the principal curvatures and the mean curvature depend on
the choice of the normal. In the above setting a unit ball has outer
unit normal $\nu$ as in \eqref{eq:surf:unit-normal} and curvatures
$\kappa_1=\kappa_2=+1$ when it is parametrised ``from left to right''
such that $x' \geq 0$, for instance by $\gamma(t) = ( -\cos t, \sin
t)$, $t \in [0,\pi]$.

Let $\phi \colon I \to \bbR$ be an angle function for $\gamma$, that
is, let $\phi(t)$ be the angle between the positive $x$-axis and the
tangent vector $\gamma'(t)$. Since $W^{2,1}_{\loc}(I)$ embeds into
$C^1_{\loc}(I)$, the angle $\phi$ can be chosen continuously in $I$ and
is then uniquely determined up to adding multiples of $2\pi$.
In terms of $\phi$, the curve $\gamma$ is characterised by fixing one
point and
\begin{equation}
  \label{eq:surf:curve-derivative-angle}
  x' = |\gamma'| \cos \phi,
  \qquad
  y' = |\gamma'| \sin \phi.
\end{equation}
The principal curvatures take the form
\begin{equation}
  \label{eq:surf:curvatures-angle}
  \kappa_1 = - \frac{\phi'}{|\gamma'|},
  \qquad
  \kappa_2 = \frac{\cos \phi}{y},
\end{equation}
and we have
\begin{equation}
  \label{eq:surf:gauss-curv-angle}
  K
  =
  - \frac{\phi' \cos \phi}{|\gamma'| y}
  =
  - \frac{\left(\sin \phi \right)'}{|\gamma'| y}
  =
  - \frac{\left( y' / |\gamma'| \right)'}{|\gamma'|y}.
\end{equation}
If $\gamma$ is parametrised with constant speed $q_\gamma>0$, we see
from \eqref{eq:surf:gauss-curv-angle} that
\begin{equation}
  \label{eq:surf:gauss-curv-integral}
  \int_{M_\gamma(J)} |K| \,d\mu
  =
  \frac{2\pi}{q_\gamma} \int_J |y''| \,d t
\end{equation}
is the $L^1$-norm of $y''$ up to a constant factor. Moreover, for such
$\gamma$ we have $\phi'^2 q_\gamma^2 = |\gamma''|^2$, and therefore
\begin{equation}
  \label{eq:surf:kappa1-int}
  \int_{M_\gamma(J)} \kappa_1^2 \,d\mu
  =
  \frac{2\pi}{q_\gamma} \int_J |\phi'|^2 y \,d t
  =
  \frac{2\pi}{q_\gamma^3} \int_J |\gamma''|^2 y \,d t
\end{equation}
is a weighted $L^2$-norm of $\phi'$ and $\gamma''$.

For a more detailed discussion of surfaces and basic geometric
analysis we refer to \cite{doCarmo76} or \cite[\S 7]{Simon83}.


\section{The models}
\label{sec:model}


\subsection{Approximate setting}
\label{sec:approximate-setting}

We study the energy \eqref{eq:intro:approx-energy} with continuous
bounded functions $H_s, k, k_G$. The precise values of $k$ and $k_G$
do not enter our arguments as long as
\eqref{eq:intro:param_restrictions} is satisfied, so we assume $k
\equiv k^\pm = 1 = -k_G^\pm \equiv -k_G$ for simplicity of notation.
Thus, our approximate energy is given by
\begin{equation}
  \label{eq:model:eps-energy}
  \calF_\eps(\gamma,u) = \calH_\eps(\gamma,u) + \calI_\eps(\gamma,u),
\end{equation}
where 
\begin{equation*}
  \calH_\eps(\gamma,u)
  =
  \int_{M_\gamma} u^2  \left( H-H_{s}(u) \right)^2  - u^2 K  \,d\mu
\end{equation*}
is the Helfrich energy of the membrane $(\gamma,u)$ and
\begin{equation*}
  \calI_{\eps}(\gamma,u)
  =
  \int_{M_\gamma} \eps |\snabla[M_\gamma]{u}|^2
  + \frac{1}{\eps} W(u)  \,d\mu
  +
  \eps \int_{M_\gamma} |B|^2  \,d\mu
\end{equation*}
the interface energy.
For $J \subset I$ we denote by $\calF_\eps(\gamma,u,J)$,
$\calH_\eps(\gamma,u,J)$, and $\calI_\eps(\gamma,u,J)$ the
corresponding integrals restricted to the set $M_\gamma(J)$.
The double well potential $W \colon \bbR \to [0,\infty)$ is a
continuous function that vanishes only in $\pm1$ and is $C^2$ around
these points; for simplicity of notation we assume that $W$ is
symmetric.

We study \eqref{eq:model:eps-energy} for rotationally symmetric
membranes $(\gamma,u) \in \calC \times \calP$, where
\begin{equation*}
  \begin{split}
  \calC
  :=
  \Big\{
    &\gamma = (x,y) \in C^{0,1}(I; \bbR^2)
    \cap W^{2,1}_{\loc}(I; \bbR^2) :\\
    &|\gamma'| = \text{const}, \;
    y(\partial I)=\set{0}, \;
    y(I) \subset (0,\infty), \;
    x' \geq 0,
    \int_{M_\gamma} |B|^2 \,d\mu < \infty, \;
    \calA_\gamma = A_0
  \Big\}
  \end{split}
\end{equation*}
and
\begin{equation*}
  \calP
  :=
  \Big\{
    u \in W^{1,1}_{\loc}(I) :
    \int_{M_\gamma} |\snabla[M_\gamma]{u}|^2 \,d\mu < \infty, \;
    \| u \|_\infty \leq C_0, \;
    \int_{M_\gamma} u \,d\mu = m A_0.
  \Big\}
\end{equation*}
The first three conditions in the definition of $\calC$ ensure that
the constant speed curve $\gamma$ generates a closed surface
$M_\gamma$. The requirement $x' \geq 0$ fixes the orientation and,
since by embedding $\gamma \in \calC$ belongs to $C^1_\loc(I;\bbR^2)$,
it guarantees that $M_\gamma$ is embedded; its main purpose is to
exclude some very non membrane-like behaviour such as infinitely many
self-intersections or zigzagging of curves in the limit.
The $L^2$-bound on $B$ and the first two conditions on the phase
fields ensure that \eqref{eq:model:eps-energy} is well-defined for
$(\gamma,u) \in \calC \times \calP$.
The uniform bound $\|u\|_\infty \leq C_0$ with a constant $C_0 \gg 1$
is more restrictive than necessary, and in many places in the proof it
can be replaced by weaker conditions such as integral bounds on $u$.
We impose the $L^\infty$ bound for convenience though, and as one
expects phase fields with small energy to be roughly between $+1$ and
$-1$ for small $\eps$, this is not a strong restriction.

The area constraints for the two lipid phases are incorporated by
prescribing the area of $M_\gamma$ and the phase integral: if in
\eqref{eq:intro:energy} the areas of the lipid phases $M^\pm$ are
$A^\pm$, then the choice $A_0 = A^+ + A^-$ and $m=(A^+-A^-)/A_0$
ensures the correct phase areas in the limit $\eps\to0$. Through the
integral constraint the set $\calP$ depends on the chosen $\gamma \in
\calC$, but since we usually consider pairs or membranes $(\gamma,u)$,
we suppress this fact in the notation.
We neglect the constraint on the enclosed volume of the membrane,
because it is not necessary for our considerations. It will become
obvious that this constraint changes continuously under the convergence
we prove and can thus be reintroduced without changes of the
arguments.

\begin{remark}
  For $\gamma \in \calC$ it is easy to see that $M_\gamma$ is a $C^1$
  surface and $M_\gamma(J)$ is a $W^{2,2}$ surface for any $J \Subset
  I$. More precisely, $\gamma=(x,y) \in C^1(\overline{I};\bbR^2) \cap
  W^{2,2}_{\loc}(I;\bbR^2)$, $y \in W^{2,1}(I)$, $\gamma'$ is
  perpendicular to the axis of revolution at $\partial I$, and these
  regularity properties cannot be improved \cite[Section 2.2]{Helmers12}.
  The energy \eqref{eq:model:eps-energy} is invariant under
  reparametrisations that preserve the orientation of $\gamma$ and the
  regularity properties of $(\gamma,u)$. In particular, if
  $(\gamma,u)$ satisfies all requirements of $\calC \times \calP$ but
  only $|\gamma'| \not=0$ instead of $|\gamma'| = \text{const}$, the
  corresponding constant speed parametrisation belongs to $\calC
  \times \calP$ and has the same energy.  Hence, considering only
  $|\gamma'| = \text{const}$ is no geometric restriction.
\end{remark}

By our assumptions \eqref{eq:intro:param_restrictions} on the bending
parameters, the calculations in \eqref{eq:intro:sec-fform-bound}
yield
\begin{equation}
  \label{eq:model:bending-positivity}
  \begin{split}
    \calH_\eps(\gamma,u,J)
    &\geq
    \int_{M_\gamma(J)} \frac{1}{2} u^2 |B|^2
    -
    u^2 H_{s}(u)^2 \,d\mu
    \\
    &\geq
    \int_{M_\gamma(J)} \frac{1}{2} u^2 |B|^2 \,d\mu
    -
    \|H_{s}\|_{\infty}^2 \|u\|_\infty^2 \calA_\gamma
  \end{split}
\end{equation}
for any $J \subset I$.
Since \eqref{eq:model:bending-positivity} provides a lower bound
for $\calH_\eps$ on $\calC \times \calP$, also $\calF_\eps$
is bounded from below. Moreover, we
have the individual bounds
\begin{align}
  \label{eq:model:modulus-energy-bound}
  |\calF_\eps(\gamma,u)|
  &\leq
  C \left( \calF_\eps(\gamma,u) + \|H_{s}\|_\infty^2 C_0^2 A_0 \right),
  \\
  \label{eq:model:phase-energy-bound}
  \calI_\eps(\gamma,u)
  &\leq
  \calF_\eps(\gamma,u) + \|H_{s}\|_\infty^2 C_0^2 A_0,
  \\ 
  \label{eq:model:second-fform-bound}
  \int_{M_\gamma} u^2 |B|^2 \,d\mu + \eps \int_{M_\gamma} |B|^2 \,d\mu
  &\leq
  C \left( \calF_\eps(\gamma,u) + \|H_{s}\|_\infty^2 C_0^2 A_0 \right)
\end{align}
for all $(\gamma,u) \in \calC \times \calP$, where $C>0$ is a generic
constant independent of $(\gamma,u)$. From
\eqref{eq:model:phase-energy-bound} and
\eqref{eq:model:second-fform-bound} we derive a bound on the first
variation of $M_\gamma$, that is, on the first variation of the area
of $M_\gamma$.

\begin{lemma}
  \label{lem:model:variation-bound}
  There is a constant $C>0$ such that
  \begin{equation*}
    \frac{1}{\sqrt{2}}
    \int_{M_\gamma} |H| \,d\mu
    \leq
    \int_{M_\gamma} |B| \,d\mu
    \leq
    C ( \calF_\eps(\gamma,u) + 1 )
  \end{equation*}
  for all $(\gamma,u) \in \calC \times \calP$.
\end{lemma}

\begin{proof}
  Splitting $M_\gamma$ into two pieces where the phase field is small
  and large, respectively, and applying H{\"o}lder's inequality, we
  get
  \begin{align*}
    \int_{M_\gamma} |B| \,d\mu
    &\leq
    \int_{M_\gamma(\set{|u| \leq 1/2})} |B| \,d\mu
    +
    \int_{M_\gamma(\set{|u| > 1/2})} |B| \,d\mu
    \\
    &\leq
    \left( \frac{1}{\eps} \calA_\gamma(\set{|u| \leq 1/2}) \right)^{1/2}
    \left( \int_{M_\gamma} \eps |B|^2 \,d\mu \right)^{1/2}
    +
    2 \sqrt{A_0} \left( \int_{M_\gamma} u^2|B|^2 \,d\mu \right)^{1/2}.
  \end{align*}
  From the interface energy we obtain the estimate
  \begin{equation*}
    \calI_\eps(\gamma,u)
    \geq
    \int_{M_\gamma(\set{|u|\leq1/2})} \frac{1}{\eps} W(u) \,d\mu
    \geq
    \left( \inf_{|u|\leq1/2} W(u) \right)
    \frac{\calA_\gamma(\set{|u|\leq1/2})}{\eps},
  \end{equation*}
  and since $W$ has a positive minimum on $[-1/2,1/2]$, we find
  \begin{equation*}
    \int_{M_\gamma} |B| \,d\mu
    \leq
    C \left( \calI_\eps(\gamma,u)^{1/2} + 1 \right)
    \left[
      \left( \int_{M_\gamma} \eps |B|^2 \,d\mu \right)^{1/2} +
      \left( \int_{M_\gamma} u^2 |B|^2 \,d\mu \right)^{1/2}
    \right].
  \end{equation*}
  The conclusion now follows from \eqref{eq:model:phase-energy-bound},
  \eqref{eq:model:second-fform-bound}, and the elementary inequality
  $\sqrt{a}+\sqrt{b} \leq \sqrt{2(a+b)} \leq \sqrt{2}
  (\sqrt{a}+\sqrt{b})$ for $a,b \geq 0$.
\end{proof}

To establish compactness, we use that the first variation of
$M_\gamma$ bounds the length of the generating curve $\gamma$. Such a
bound is for instance deduced from the well-known bound on the
intrinsic diameter from \cite{Topping08} or for surfaces of revolution
easily proved by an integration by parts as in \cite[Section
2.3]{Helmers12}.

\begin{lemma}
  \label{lem:surf:length-bound}
  Let $\gamma = (x,y) \in C^{0,1}(I;\bbR^2) \cap
  W^{2,1}_{\loc}(I;\bbR^2)$ be a curve such that $y(I) \subset
  (0,\infty)$, $y(\partial I) = \set{0}$. Then
  \begin{equation*}
    \int_{M_\gamma} |H| d\mu
    \geq
    2\pi \calL_\gamma.
  \end{equation*}
\end{lemma}

\begin{remark}
  Combining Lemmas \ref{lem:model:variation-bound} and
  \ref{lem:surf:length-bound}, we see that any sequence
  $(\gamma_\eps,u_\eps) \in \calC \times \calP$ with uniformly bounded
  energy has uniformly bounded length.
  For this conclusion we could have argued with $\int u^2 H^2 + \eps
  H^2 \,d\mu$ directly, instead of using the second fundamental form
  in the proof of Lemma \ref{lem:model:variation-bound}.
  The following example shows that an additional energy term like
  $\eps \int H^2 \,d\mu$ or $\eps \int |B|^2 \,d\mu$ is necessary to
  obtain the length bound.
  Let $M_\eps$ be a sequence of ``dumbbells'' that consist of two
  spheres, which are smoothly connected by a cylinder of length
  $l_\eps$ and diameter $h_\eps$, and let the phase field $u_\eps$ be
  $0$ on the cylinder and $+1$ and $-1$ on the spheres with exactly one
  transition with gradient of order $\eps^{-1}$ at each connection.
  Then $\calH_\eps \sim 0$ on the cylinder and $\calH_\eps$ is bounded
  independently of $\eps$ on the spheres. The contribution of $u_\eps
  \sim 0$ on the cylinder and of the two phase transitions stems from
  $\int \eps|\snabla[M_\eps]{u_\eps}|^2 + \frac{1}{\eps} W(u_\eps)
  \,d\mu_\eps$ and is of order $\frac{1}{\eps} l_\eps h_\eps +
  h_\eps$, and the smoothing of the connections between the cylinder
  and the spheres can be done where $u_\eps \sim 0$; see Section
  \ref{sec:proof:ub} for the details of the construction of a recovery
  sequence.
  Thus, if $l_\eps \to \infty$ and $h_\eps \to 0$ such that $l_\eps
  h_\eps \sim \eps$, the energy without $\eps \int |B|^2 \,d\mu$ is
  bounded, but the length of the generating curve is unbounded as
  $\eps\to0$; $(\gamma_\eps, u_\eps)$ can easily be made admissible
  for some $A_0$ and $m$, since the area and phase constraint, which
  are disturbed by the vanishing cylinder, can be recovered by
  slightly perturbing the spheres.
  On the other hand, $\eps \int |B|^2 \,d\mu \sim \eps
  l_\eps/h_\eps$ on the cylinder, and therefore $l_\eps \to
  \infty$ is excluded by a uniform bound on $\calF_\eps$.

  The scaling of $\eps$ in the stabilising term is critical. If the
  energy contains $\eps^p \int |B|^2 \,d\mu$ with $p>1$, the above
  example still works and there is no length bound. If $p<1$, tangent
  discontinuities in the limit are excluded, since an argument similar
  to the proof of Lemma~\ref{lem:model:variation-bound} yields an
  $L^q$-bound for some $q>1$ on the second fundamental form and thus
  on $\kappa_1$; compare the equi-coercivity arguments in
  Section~\ref{sec:proof:equi-coercivity}.
\end{remark}


\subsection{Limit setting}

The major technical difficulties in the limit of
\eqref{eq:model:eps-energy} as $\eps \to 0$ stem from the appearance of
kinks and from the axis of revolution. In particular, at the axis the
compactness result for sequences $(\gamma_\eps,u_\eps)$ and the
regularity properties of the limit are weaker than elsewhere.
Limit curves will have parametrisations in
\begin{equation*}
  \begin{split}
    \calD
    := \Big\{
    &\gamma = (x,y) \in C^{0,1}(I; \bbR^2)
    :
    \\
    &|\gamma'| \equiv q_\gamma = \text{const in} \set{y>0}, \;
    |\gamma'| = x' \leq q_\gamma \text{ in} \set{y=0},
    \\
    &y(\partial I) = \set{0}, \;
    y \geq 0, \;
    x' \geq 0
    \\
    &\text{there is } S_\gamma \subset \set{y>0}
    \text{ s.\,t. } \calH^1(M_\gamma(S_\gamma)) <
    \infty \text{ and} 
    \\
    &\int_{M_\gamma(\set{y>0} \sm S_\gamma)} |B|^2 \,d\mu < \infty, \;
    \calA_\gamma = A_0
    \Big\}.
  \end{split}
\end{equation*}
A curve $\gamma \in \calD$ is globally Lipschitz with Lipschitz
constant $q_\gamma$, and its restriction to $\set{y>0}$ is a constant
speed parametrisation. Since $\set{y>0} \subset \bbR$ is open, it is
the union of its countably many connected components, which are
disjoint intervals. In a slight abuse of language we refer to a
component $\omega$ of $\set{y>0}$ also as component of $\gamma$ and
call $M_\gamma(\omega)$ a component of $M_\gamma$. Thus, $M_\gamma$
consists of at most countably many components, which are connected
through the axis of revolution.

Due to $\calH^1(M_\gamma(S_\gamma)) < \infty$ the set $S_\gamma \cap
J$ is finite for any $J \Subset \set{y>0}$, and since $S$ can be
written as countable union of such sets it is countable.
The bound on the second fundamental form yields $\gamma \in W^{2,2}(J
\sm S_\gamma; \bbR^2)$. By embedding into $C^1(\overline{J}; \bbR^2)$
the tangent vector $\gamma'$ is continuous from either side at any $s
\in S_\gamma$, that is, $S_\gamma$ indeed contains the tangent
discontinuities of $\gamma$ in $\set{y>0}$.

In contrast to $\calC$, a component $M_\gamma(\omega)$ of $M_\gamma$,
$\gamma \in \calD$ is embedded only between adjacent kinks, but in
general not globally. Moreover, if kinks accumulate at $a
\in \partial\omega$, the limit of $\gamma'(t)$ as $t \to a$, $t \in
\omega$ need not exist; $\gamma'$ is perpendicular to the axis of
revolution in the following weak sense.

\begin{lemma}
  \label{lem:model:weak-continuity}
  Let $\omega = (a,b)$ be a component of $\gamma=(x,y) \in \calD$ and
  assume that $(s_j) \subset S_\gamma \cap \omega$ is a decreasing
  sequence such that $s_j \to a$ as $j \to \infty$ and $\gamma \in
  W^{2,2}(s_{j+1},s_j)$ for all $j \in \bbN$. Then the one-sided
  approximate limit of $x'$ vanishes at $a$, that is
  \begin{equation*}
    \lim_{\rho \searrow 0} \frac{1}{\rho} \int_a^{a+\rho} x' \,d t = 0,
  \end{equation*}
  and $|y'|$ has one-sided approximate limit $q_\gamma$.
  Moreover, $\gamma$ is almost piecewise straight near $a$ in the
  sense
  \begin{equation*}
    \lim_{j \to \infty} \osc_{(s_{j+1},s_j)} \gamma'
    =
    \lim_{j \to \infty} \sup_{t,s \in (s_{j+1},s_j)} |\gamma'(t) - \gamma'(s)|
    =
    0.
  \end{equation*}
\end{lemma}

\begin{proof}
  Lipschitz continuity of $y$ implies $y(t) \leq y(a) + q_\gamma (t-a)
  \leq q_\gamma \rho$ in $(a,a+\rho)$, hence using
  \eqref{eq:surf:curvatures} we conclude
  \begin{equation*}
    \frac{1}{\rho} \int_a^{a+\rho} x'^2 \,d t
    \leq
    q_\gamma^2 \int_a^{a+\rho} \frac{x'^2}{q_\gamma y} \,d t
    \leq
    \frac{q_\gamma^2}{2\pi}
    \int_{M_\gamma((a,a+\rho) \sm S_\gamma)} \kappa_2^2 \,d\mu.
  \end{equation*}
  The right hand side tends to zero as $\rho \to 0$, because the
  $L^2$-norm of the second fundamental form of $M_\gamma(\omega \sm
  S_\gamma)$ is finite. The approximate limit $q_\gamma$ of $|y'|$
  follows from $y'^2 = q_\gamma^2 - x'^2$ almost everywhere in
  $\omega$.

  For the straightness recall \eqref{eq:surf:gauss-curv-integral} and
  $|B|^2 \geq |K|$ almost everywhere in $\omega$, thus
  \begin{equation*}
    \sum_{j=1}^\infty \int_{s_{j+1}}^{s_j} |y''| \,d t
    \leq
    \int_{M_\gamma(\omega \sm S_\gamma)} |B|^2 \,d\mu
    <
    \infty.
  \end{equation*}
  Consequently,
  \begin{equation*}
    \sup_{t,s \in (s_{j+1},s_j)} |y'(t)-y'(s)|
    \leq \int_{s_{j+1}}^{s_j} |y''| \,d t
    \to 0
    \qquad\text{as } j \to \infty,
  \end{equation*}
  and likewise for $x'$ due to $x'\geq0$ and
  \begin{equation*}
    |x'(t)-x'(s)|^2
    \leq
    |x'^2(t)-x'^2(s)|
    =
    |y'^2(t)-y'^2(s)|
    \leq
    2 q_\gamma |y'(t)-y'(s)|.
    \qedhere
  \end{equation*}
\end{proof}

According to Lemma \ref{lem:model:weak-continuity}, $\gamma$ consists
roughly of straight line segments when approaching the component
boundary, but the directions of these segments may vary as long as the
approximate limit is perpendicular to the axis of revolution.
If kinks do not accumulate near a component boundary, then, as for
$\gamma \in \calC$, the classical limit tangent exists and is
perpendicular to the axis of revolution.
An example for a curve in $\calD$ is given in Figure
\ref{fig:model:curve-in-d}.

\begin{figure}
  \centering
  \includegraphics[width=.9\textwidth]{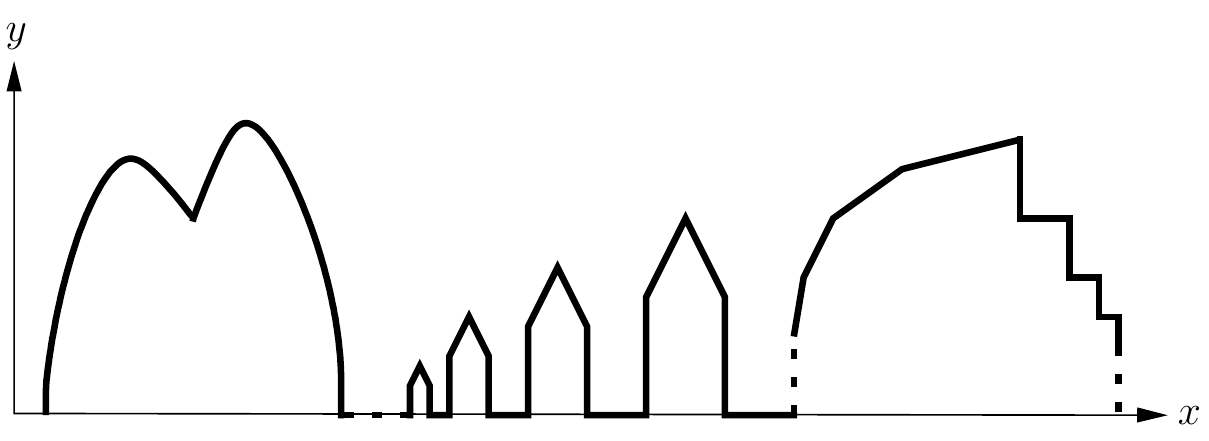}
  \caption{Example of a curve $\gamma \in \calD$. The component on the
    left is regular except for one kink. In the right component kinks
    accumulate at both ends, where at the left end the limit tangent
    exist, but at the right end it does not. In the centre there are
    countably many self-similar components $\omega_k$ decreasing from
    right to left such that one can easily find a scaling of
    $\gamma(\omega_k)$ and $\calL_\gamma(\omega_k)$ that leaves $B$
    bounded in $L^2$. }
  \label{fig:model:curve-in-d}
\end{figure}

To $\gamma \in \calD$ we associate a phase field $u$ in
\begin{equation*}
  \begin{split}
    \calQ
    := \Big\{
      &u \colon I \to [-C_0,C_0] :
      u \in \set{\pm 1} \text{ piecewise constant in} \set{y>0},
      \\
      &\int_{M_\gamma} u \,d\mu
      = m A_0,\; \calH^1(M_\gamma(S_u)) < \infty
    \Big\}.
  \end{split}
\end{equation*}
Here $S_u \subset \set{y>0}$ denotes the jump set of $u$, and we call
$s \in S_u$ and the corresponding circle $M_\gamma(\set{s})$ an
interface of $(\gamma,u)$.
The set $\calQ$ resembles the set of special functions of bounded
variation SBV with values in $\set{\pm1}$, weighted with the height
$y$ of the generating curve $\gamma = (x,y) \in \calD$. Indeed, for $u
\in \calQ$ and any $J \Subset \set{y>0}$ we have $u \in SBV(J;
\set{\pm1})$, but jumps of height $2$ may accumulate near the axis of
revolution and $u$ is not specified in $\set{y=0}$.
We emphasise that in our notation $S_u$ and $S_\gamma$ are subsets of
$\set{y>0}$, because kinks and interfaces on the axis of revolution do
not contribute to the limit energy defined below.
Moreover, kinks are not restricted to interfaces, that is, there may
be points $s \in S_\gamma \sm S_u$. We call such points {\em ghost
  interfaces}, as opposed to {\em proper interfaces}, since their
contribution to the limit energy is concentrated on lines as for
interfaces and contains the interface energy of the latter.

For $(\gamma,u) \in \calD \times \calQ$ we consider the energy
$\calF = \calH + \calI$ with Helfrich energy
\begin{equation*}
  \calH(\gamma,u)
  =
  \int_{M_\gamma(\set{y>0} \sm S_\gamma)} (H-H_{s}(u))^2 - K \,d\mu
\end{equation*}
and interface energy
\begin{equation*}
  \calI(\gamma,u)
  =
  2\pi \sum_{s \in S_\gamma \cup S_u} \left(
    \sigma + \hat\sigma |[\gamma'](s)|
  \right) y(s)
  +
  2\pi \hat\sigma \calL_\gamma(\set{y=0});
\end{equation*}
as before, $\calF(\cdot,\cdot,J)$, $\calH(\cdot,\cdot,J)$, and
$\calI(\cdot,\cdot,J)$ denote the restrictions to $J \subset I$.
Recall that $\sigma,\hat\sigma$ are given by
\eqref{eq:intro:sigma-sigma-hat} and that $|[\gamma'](s)|$ denotes the
modulus of the angle enclosed by the two one-sided tangent vectors at
$s$ modulo $2\pi$, that is, the jump of the tangent vector because its
length is fixed. The size of $\set{y=0}$ appears in $\calI$, because
it stems from the second fundamental form in $\calI_\eps$, and
$\set{y=0}$ might be interpreted as a (ghost) interface between
components of $M_\gamma$.
As for $\calF_\eps$, we find
\begin{equation*}
  \calH(\gamma,u)
  \geq
  \frac{1}{2} \int_{M_\gamma(\set{y>0} \sm S_\gamma)} |B|^2 \,d\mu
  - \|H_{s}\|_\infty^2 \calA_\gamma
\end{equation*}
and bounds corresponding to
\eqref{eq:model:modulus-energy-bound}--\eqref{eq:model:second-fform-bound}.
Moreover, also $\calF$ is invariant under reparametrisations that
preserve orientation and regularity properties.


\subsection{Convergence theorem}

We extend $\calF_\eps$ and $\calF$ to $C^0(I; \bbR^2) \times L^1(I)$
by setting $\calF_\eps(\gamma,u) = \calF(\gamma,u) = \infty$ whenever
$(\gamma,u)$ does not belong to $\calC \times \calP$ or
$\calD \times \calQ$, respectively.
The main result of this paper is the following theorem.

\begin{theorem}
  \label{thm:model:gamma-type-conv}
  The energies $\calF_\eps$ are equi-coercive, that is, any sequence
  $((\gamma_\eps,u_\eps)) \subset \calC \times \calP$ with uniformly
  bounded energy $\calF_\eps(\gamma_\eps,u_\eps)$ admits a subsequence
  that converges in $C^0(I; \bbR^2) \times L^1(\set{y>0})$ to some
  $(\gamma,u) \in \calD \times \calQ$.
  Furthermore, $\calF_\eps$ converges to $\calF$ in the following sense:
  \begin{itemize}
  \item any sequence $((\gamma_\eps,u_\eps)) \subset C^0(I;\bbR^2)
    \times L^1(I)$ that converges to $(\gamma,u)$ in $C^0(I; \bbR^2)
    \times L^1(\set{y>0})$ satisfies the lower bound inequality
    \begin{equation*}
      \liminf_{\eps \to 0} \calF_\eps(\gamma_\eps,u_\eps)
      \geq
      \calF(\gamma,u);
    \end{equation*}
  \item for any $(\gamma,u) \in \calD \times \calQ$ such that $\gamma$
    is parametrised with constant speed almost everywhere in $I$ there
    exists a recovery sequence $((\gamma_\eps,u_\eps)) \subset \calC
    \times \calP$ that converges to $(\gamma,u)$ in $C^0(I; \bbR^2)
    \times L^1(\set{y>0})$ and satisfies the upper bound inequality
    \begin{equation*}
      \limsup_{\eps \to 0} \calF_\eps(\gamma_\eps,u_\eps)
      \leq
      \calF(\gamma,u).
    \end{equation*}
  \end{itemize}
\end{theorem}

Theorem \ref{thm:model:gamma-type-conv} differs from
$\Gamma$-convergence in two aspects. First, the underlying convergence
of the phase fields $u_\eps$ in $L^1(\set{y>0})$ depends on the limit
curve $\gamma=(x,y)$, because there is in general insufficient control
on $u_\eps$ in $\set{y=0}$. Of course, due to $\|u_\eps\|_{\infty}\leq
C_0$ we could extract a weakly-$\star$ convergent subsequence, but
since the $L^\infty$-bound is artificial and the value of the limit
$u$ in $\set{y=0}$ is not used by $\calF$, we prefer the above
setting, where the limit phase field is essentially undefined in
$\set{y=0}$.
Second, in the upper bound inequality we construct a recovery sequence
only for limits $(\gamma,u)$ with constant speed $|\gamma'|$ in all of
$I$. Nevertheless, as for $\Gamma$-convergence it is true that almost
minimising sequences for $\calF_\eps$ cluster only in minimisers of
$\calF$.

\begin{corollary}
  Let $(\gamma_\eps,u_\eps) \in \calC \times \calP$ converge to
  $(\gamma,u) \in \calD \times \calQ$ in $C^0(I; \bbR^2) \times
  L^1(\set{y>0})$ such that $\calF_\eps(\gamma_\eps,u_\eps) = \inf
  \calF_\eps + o(1)_{\eps\to0}$.
  Then $(\gamma,u)$ minimises $\calF$ in $\calD \times \calQ$.
\end{corollary}

\begin{proof}
  Given an arbitrary $(\widetilde \eta, \widetilde w) \in \calD \times
  \calQ$, $\widetilde\eta = (x_{\widetilde\eta}, y_{\widetilde\eta})$,
  a constant speed parametrisation of the membrane represented by
  $(\widetilde \eta, \widetilde w)$ is found as in Section \ref{sec:surf}
  and the first remark in Section \ref{sec:approximate-setting}:
  First, removing constancy intervals of $\widetilde \eta$ in
  $\set{\widetilde y=0}$ does not change the membrane, its area,
  energy $\calF$, or phase integral. Then, if $\widetilde \eta$ has no
  constancy intervals and $I=(a,b)$, the function
  \begin{equation*}
    \psi(t)
    =
    a + \frac{b-a}{\calL_{\widetilde \eta}}
    \int_a^t |\widetilde\eta'(s)| \,d s
  \end{equation*}
  is strictly increasing, and the parametrisation $(\eta,w) =
  (\widetilde \eta \circ \psi^{-1}, \widetilde w \circ \psi^{-1})$,
  $\eta = (x_\eta, y_\eta)$ has constant speed $\calL_{\widetilde
    \eta}/(b-a)$. Since $\psi$ is affine in each component of
  $\widetilde\eta$, the pair $(\eta,w)$ inherits its differentiability
  properties and bounds in $\set{y_\eta>0}$ from $(\widetilde
  \eta,\widetilde w)$ in $\set{y_{\widetilde\eta}>0}$.  Again, energy,
  area and phase integral are unchanged.

  With a recovery sequence $(\eta_\eps, w_\eps)$ for $(\eta,w)$ we now
  obtain
  \begin{equation*}
    \calF(\gamma,u)
    \leq
    \liminf_{\eps \to 0} \calF_\eps(\gamma_\eps,u_\eps)
    =
    \liminf_{\eps \to 0} \left( \inf \calF_\eps \right)
    \leq
    \limsup_{\eps \to 0} \calF_\eps(\eta_\eps, w_\eps)
    \leq
    \calF(\eta, w)
    =
    \calF(\widetilde \eta, \widetilde w),
  \end{equation*}
  and by arbitrariness of $(\widetilde \eta, \widetilde w)$ we
  conclude that $(\gamma,u)$ has minimal energy $\calF$.
\end{proof}

The relation between $\calF$ and $\Gamma$-$\lim \calF_\eps$ is
discussed further in Section \ref{sec:some-generalisations}.


\subsection{Numerical examples}
\label{sec:numerical-examples}

Since it is defined for relatively smooth membranes, $\calF_\eps$ is
better suited for numerical simulation than $\calF$ and can be
compared to an approximation of a limit without kinks given by
\begin{equation}
  \label{eq:model:simple-approx-energy}
  \calE_\eps(\gamma,u)
  =
  \int_{M_\gamma} (H-H_s(u))^2 - K \,d\mu
  +
  \int_{M_\gamma} \eps |\snabla[M_\gamma]{u}|^2 + \frac{1}{\eps} W(u) \,d\mu
\end{equation}
for $(\gamma,u) \in \calC \times \calP$. The energy
\eqref{eq:model:simple-approx-energy} has been studied in numerical
simulations and by means of formal asymptotic expansion for arbitrary
smooth surfaces, for instance in \cite{ElSt10_Modeling,ElSt10}. For
rotationally symmetric membranes the $\Gamma$-limit of
\eqref{eq:model:simple-approx-energy} is given by
\begin{equation}
  \label{eq:model:simple-limit-energy-special}
  \calE(\gamma,u)
  =
  \int_{M_\gamma} (H-H_s(u))^2 - K \,d\mu + \sigma \calH^1(M_\gamma(S_u))
\end{equation}
on membranes $(\gamma,u)$ such that $M_\gamma$ is a topological
sphere \cite{Helmers12}.

For numerical illustrations we consider a gradient flow type evolution
for $\calE_\eps$ and $\calF_\eps$ that consists of an $L^2$ flow for
the surface and a weighted $L^2$ flow for the phase field; the
constraints are incorporated by Lagrange multipliers. This flow has to
our knowledge first been studied in \cite{ElSt10_Modeling,ElSt10},
where the derivation of the flow equations is presented in full
detail. The numerical results below were obtained by incorporating the
phase field into the scheme for rotational symmetric surfaces flows
from \cite{MaSi02}.

\begin{figure}[t]
  \centering
  \includegraphics{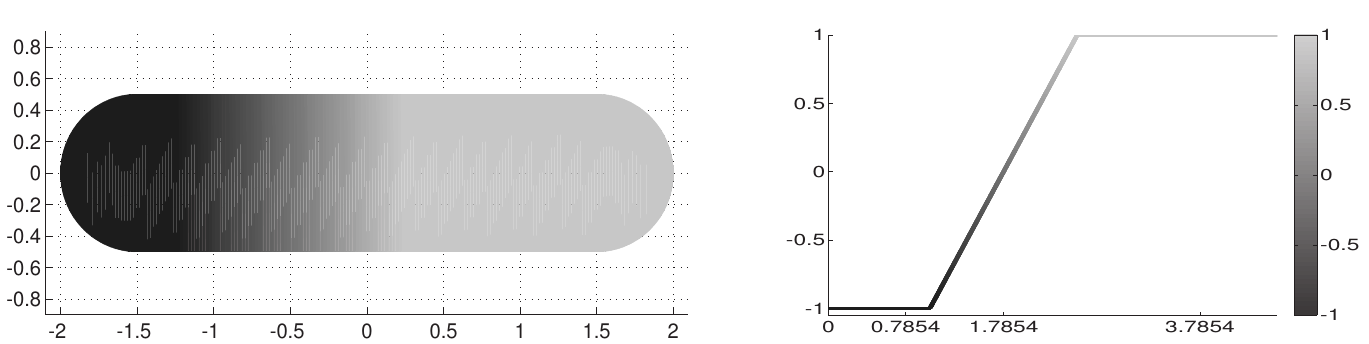}
  \caption{Initial data for the numerical examples in Section
    \ref{sec:numerical-examples}, cross section of the surface on the
    left, phase field over arc length on the right. The marks on the
    horizontal axis indicate the interface and the connection of
    spherical caps and cylinder.}
  \label{fig:initial-data}
\end{figure}

The initial data for the simulations below is shown in Figure
\ref{fig:initial-data}. The surface is a cylinder of length $3$ and
radius $1/2$ with spherical caps; it is centred in the origin so that
the $x$-coordinate ranges from $-2$ to $2$. The initial phase field is
\begin{equation*}
  u(x,y) =
  \begin{cases}
    -1                          &\text{if } x \leq -\frac{5}{4}, \\
    \frac{4}{3} x + \frac{2}{3} &\text{if } -\frac{5}{4} < x < \frac{1}{4}, \\
    +1                          &\text{if } \frac{1}{4} \leq x.
  \end{cases}
\end{equation*}
The spontaneous curvature $H_s(u)$ is the fifth-order polynomial
interpolation of $H_s(1)=2$, $H_s(-1)=1$, $H_s'(\pm1) = H_s''(\pm1) =
0$ in $[-1,1]$ and extended constantly to the whole real line.

Figure \ref{fig:kink} shows the numerically stationary membranes and
the angle between their generating curves and the positive $x$-axis for
$\calE_\eps$ and $\calF_\eps$ with $\eps=0.05$.
Obviously, while there is a smooth and rather ample neck region for
$\calE_\eps$, the curve for $\calF_\eps$ makes a sharp turn: the angle
almost jumps from about $-0.5$ to $1.05$, and the neck region is
limited to a small neighbourhood of the approximate kink, which
compares well with the experimental observations in \cite{BaHeWe03}.
Also, the light phase $u_\eps=1$ of the membrane is closer to a round
sphere for $\calF_\eps$ than for $\calE_\eps$.

\begin{figure}[ht]
  \centering
  \includegraphics[width=\textwidth]{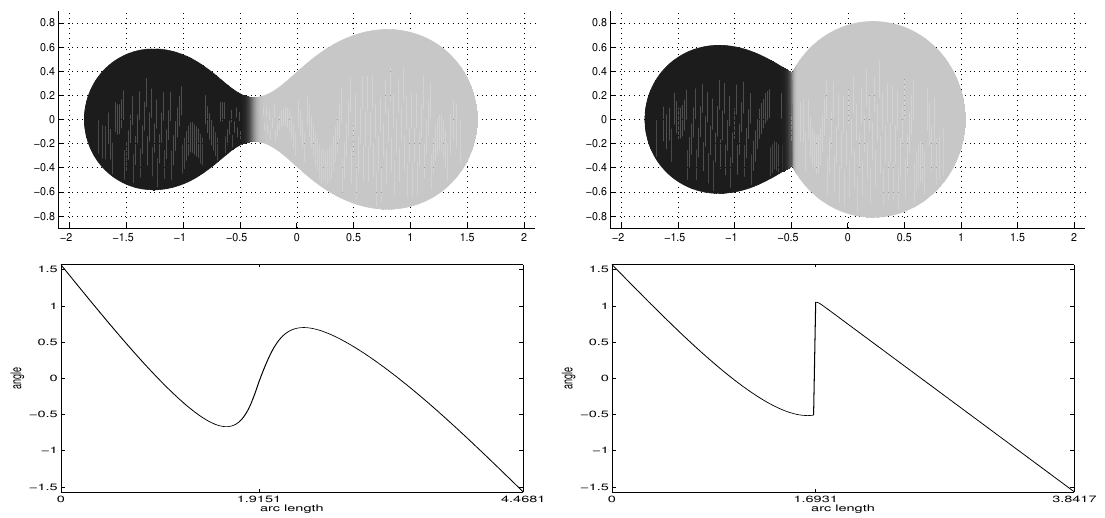}
  \caption{Numerically stationary shapes and angle between the
    generating curves and the positive $x$-axis over arc length for
    $\calE_\eps$ on the left and $\calF_\eps$ on the right.}
  \label{fig:kink}
\end{figure}

A different behaviour can be seen in Figure \ref{fig:no-kink-surface},
which shows the numerically stationary membrane for the energy
$\widetilde\calF_\eps$ that differs from $\calF_\eps$ in that no Gauss
curvature is present; we will discuss in Section
\ref{sec:gauss-curvature-axis} that our theorem can be adapted to this
case. The stationary shape for the corresponding energy
$\widetilde\calE_\eps$ is the same as for $\calE_\eps$ in Figure
\ref{fig:kink}, because the Gauss curvature integral in $\calE_\eps$
is a topological invariant.
One can see that the neck for $\widetilde\calF_\eps$ has smaller
diameter than for $\widetilde\calE_\eps$, there is, however, no kink,
and the neck region is as ample as for $\widetilde \calE_\eps$.

\begin{figure}
  \centering
  \includegraphics[width=.46\textwidth]{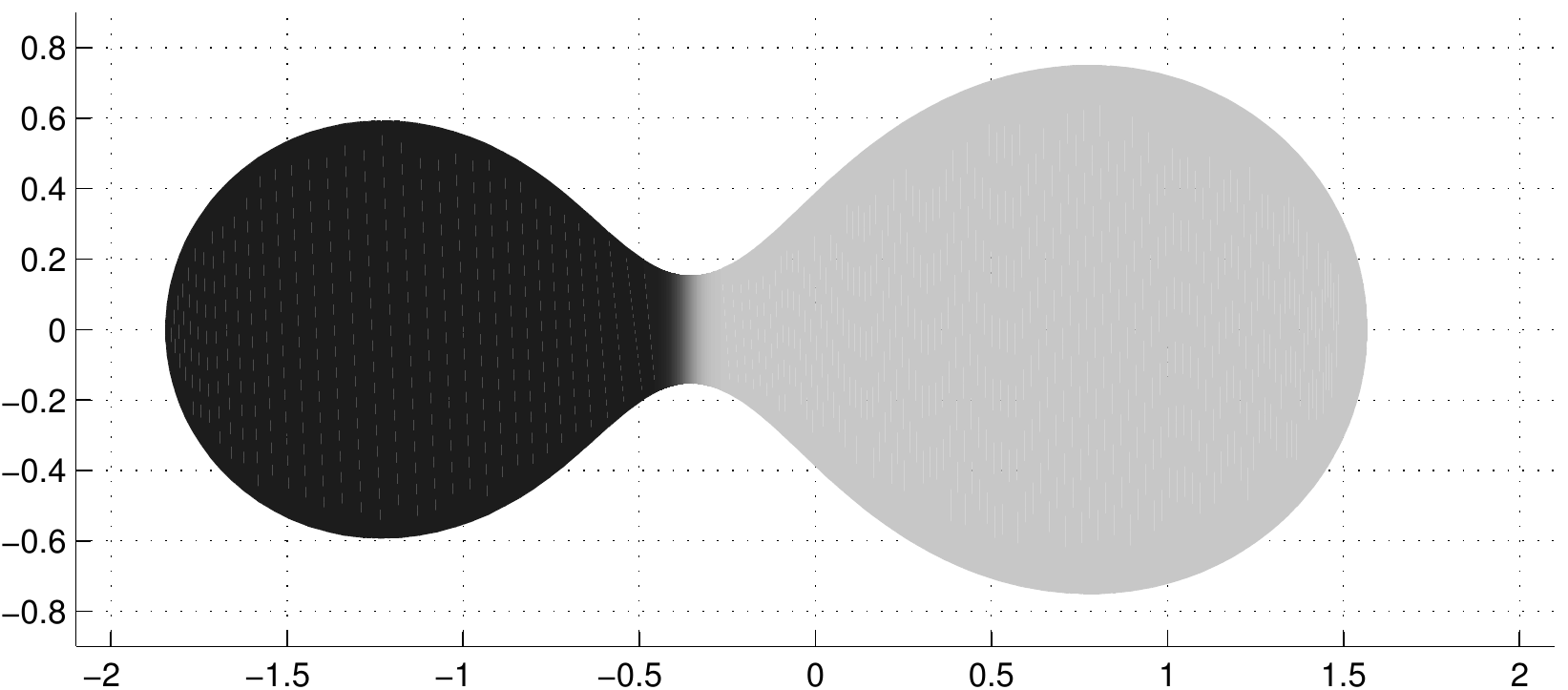}
  \caption{Numerically stationary shape for $\widetilde\calF_\eps$.}
  \label{fig:no-kink-surface}
\end{figure}


\section{Proof of Theorem \ref{thm:model:gamma-type-conv}}
\label{sec:proof}

The proof of Theorem \ref{thm:model:gamma-type-conv} follows the ideas
of \cite{Helmers11}, where we studied a one-dimensional analogue of
two-phase membranes, and is split into the three steps
equi-coercivity, lower bound, and upper bound inequality.  In the
following we write $M_\eps$ instead of $M_{\gamma_\eps}$ and so forth
when considering sequences of membranes. If convenient for
clarification, we also add an index $\gamma$ or $\eps$ to other
quantities such as $\mu$, $H$, and so on.


\subsection{Equi-coercivity}
\label{sec:proof:equi-coercivity}

\begin{lemma}
  \label{lem:proof:equi-coercivity}
  Let $((\gamma_\eps,u_\eps)) \subset \calC \times \calP$ be a
  sequence with uniformly bounded energy
  $\calF_\eps(\gamma_\eps,u_\eps)$.
  Then there exist $(\gamma,u) \in \calD \times \calQ$, $\gamma =
  (x,y)$, a countable set $S \subset \set{y>0}$ with $S_\gamma \cup
  S_u \subset S$ and $S \cap J$ finite for any $J \Subset \set{y>0}$,
  and a subsequence, not relabelled, such that
  \begin{itemize}
  \item $\gamma_\eps \wsto \gamma$ in $W^{1,\infty}(I; \bbR^2)$;
  \item $u_\eps \to u$ in $L^p(\set{y>0})$ for any $p \in [1,\infty)$;
  \item $\gamma_\eps \wto \gamma$ in $W^{2,2}_{\loc}(\set{y>0} \sm S;
    \bbR^2)$;
  \item in any $J \Subset \set{y>0} \sm S$ there holds $|u_\eps| \geq
    1/2$ for all sufficiently small $\eps$.
  \end{itemize}
\end{lemma}

\begin{proof}
  Let $\gamma_\eps = (x_\eps,y_\eps)$ and $|\gamma_\eps'| =
  q_\eps$.
  With Lemma \ref{lem:model:variation-bound}, Lemma
  \ref{lem:surf:length-bound}, and H{\"o}lder's inequality we find
  \begin{equation*}
    2 \pi q_\eps |I|
    =
    2 \pi \calL_\eps
    \leq
    \left( \calA_\eps \int_{M_\eps} H^2 \,d\mu \right)^{1/2},
  \end{equation*}
  thus the sequence $(q_\eps)$ is uniformly bounded from above.
  Since translations in $x$-direction do not change the energy, we may
  assume that all $\gamma_\eps$ have a common end point. Hence,
  $(\gamma_\eps)$ is bounded in $W^{1,\infty}(I; \bbR^2)$ and we may
  extract a subsequence such that $q_\eps \to q$ and $\gamma_\eps
  \wsto \gamma = (x,y)$ in $W^{1,\infty}(I;\bbR^2) = C^{0,1}(I;\bbR)$.
  In particular, $y \geq 0$, $y(\partial I) = \set{0}$, and the
  convergence of $(\gamma_\eps)$ is uniform in $\overline{I}$. Since
  the set of non-negative functions is closed under weak-$\star$
  convergence in $L^\infty(I)$, $\gamma$ satisfies $x'\geq0$.
  From
  \begin{equation*}
    A_0 = \calA_\eps = 2\pi q_\eps \int_I y_\eps \,d t
    \to
    2\pi q \int_I y \,d t
  \end{equation*}
  we conclude that neither $q=0$ nor $y \equiv 0$ in $I$. Without loss
  of generality, we assume $q=1$, thus $|\gamma'| \leq 1$ almost
  everywhere in $I$.

  Uniform convergence implies that for any $J \Subset \set{y>0}$
  there is a constant $c_J>0$ such that $y_\eps \geq c_J$ in $J$ for
  all sufficiently small $\eps$. Therefore,
  \begin{equation}
    \label{eq:proof:coerc-phase-fields}
    \frac{1}{2\pi}
    \int_{M_\eps(J)} \eps |\snabla[M_\eps]{u_\eps}|^2 + \frac{1}{\eps}
    W(u_\eps) \,d\mu_\eps
    \geq
    c_J \int_{J} \frac{\eps}{q_\eps} |u_{\eps}'|^2 +
    \frac{q_\eps}{\eps} W(u_{\eps}) \,d t
  \end{equation}
  and the well-known arguments of Modica and Mortola
  \cite{Modica87,MoMo77} apply in $J$, see in particular \cite[Lemma
  6.2 and Remark 6.3]{Braides02} for a proof in one dimension.
  The outcome is a finite set of points $S_J \subset J$ and a
  piecewise constant function $u \colon J \to \set{\pm1}$ whose jump
  set is contained in $S_J$ such that a subsequence of $u_\eps$
  converges to $u$ in measure and almost everywhere in $J \sm
  S_J$. Since $(u_\eps)$ is uniformly bounded in $L^\infty(I)$,
  convergence in $L^p(J)$ for any $p \in [1,\infty)$ follows.
  Moreover, in the one-dimensional setting we obtain that in any set
  compactly contained in $J \sm S_J$ we have $|u_\eps| \geq 1/2$ for
  all sufficiently small $\eps$.

  Exhausting $\set{y>0}$ by a sequence of increasing sets such as $J_k
  = \set{y>1/k}$ as $k \to \infty$ and taking a diagonal sequence, we
  find an at most countable set $S \subset \set{y>0}$ with $S \cap J$
  finite for any $J \Subset \set{y>0}$, a function $u \colon \set{y>0}
  \to \set{\pm1}$ with $S_u \subset S$, and subsequence of $(u_\eps)$
  that converges to $u$ in measure and almost everywhere in
  $\set{y>0}$ and that satisfies $|u_\eps| \geq 1/2$ in any $J \Subset
  \set{y>0} \sm S$ for all sufficiently small $\eps>0$ depending on
  $J$. From the uniform $L^\infty$-bound on $u_\eps$ we infer
  convergence in $L^p(\set{y>0})$ for any $p \in [1,\infty)$.

  Along this subsequence we establish further compactness of the
  curves. Given $J \Subset \set{y>0} \sm S$, there holds $q_\eps \leq
  2$, $|u_\eps| \geq 1/2$, and $y_\eps \geq c_J$ for all sufficiently
  small $\eps>0$. Therefore, using \eqref{eq:surf:kappa1-int} we find
  that
  \begin{equation*}
    \int_{M_\eps} u_\eps^2 |B_\eps|^2 \,d\mu_\eps
    \geq
    \frac{1}{4} \int_{M_\eps(J)} \kappa_{1,\eps}^2 \,d\mu_\eps
    \geq
    \frac{\pi c_J}{16} \int_J |\gamma_\eps''|^2 \,d t
  \end{equation*}
  is uniformly bounded for all sufficiently small $\eps$, and a
  subsequence of $\gamma_\eps''$ converges weakly to some $\gamma_J''$
  in $L^2(J; \bbR^2)$. From $\gamma_\eps \wsto \gamma$ in
  $W^{1,\infty}(I; \bbR^2)$ we infer that $\gamma''_J$ is the
  weak derivative of $\gamma'$ in $J$ and that the whole sequence
  converges weakly in $W^{2,2}(J; \bbR^2)$.
  This shows $\gamma_\eps \wto \gamma$ in $W^{2,2}_{\loc}(\set{y>0}
  \sm S; \bbR^2)$ and by embedding $\gamma_\eps \to \gamma$ in
  $C^1_{\loc}(\set{y>0} \sm S; \bbR^2)$ and $\gamma_\eps' \to \gamma'$
  pointwise in $\set{y>0} \sm S$. Hence, we obtain $S_\gamma \subset
  S$, $1 = \lim q_\eps^2 = \lim |\gamma_\eps'|^2 = |\gamma'|^2$ in
  $\set{y>0} \sm S$, and
  \begin{equation*}
    A_0
    =
    \calA_\eps
    =
    2\pi \int_{\set{y>0}} q_\eps y_\eps \,d t
    +
    2\pi \int_{\set{y=0}} q_\eps y_\eps \,d t
    \to
    2\pi \int_{\set{y>0}}  y \,d t
    =
    \calA_\gamma
  \end{equation*}
  as well as
  \begin{equation*}
    m A_0 = \int_{M_\eps} u_\eps \,d\mu_\eps
    \to
    \int_{M_\gamma} u \,d\mu.
  \end{equation*}

  To conclude $(\gamma,u) \in \calD \times \calQ$ we must show that
  $\calH^1(M_\gamma(S_u \cup S_\gamma))$ and
  $\int_{M_\gamma(\set{y>0}\sm S_\gamma)} |B|^2 \,d\mu$ are finite.
  From $\gamma_\eps \wto \gamma$ in $W^{2,2}(J; \bbR^2)$, $u_\eps \to
  u \in \set{\pm1}$ in $L^2(J)$, and $\sup_\eps
  \|u_\eps\|_{\infty}<\infty$ in any $J \Subset \set{y>0} \sm S$ we
  infer that
  \begin{equation*}
    u_\eps \kappa_{1,\eps} \sqrt{q_\eps y_\eps} \wto u \kappa_1 \sqrt{y}
    \qquad\text{and}\qquad
    u_\eps \kappa_{2,\eps} \sqrt{q_\eps y_\eps} \to u \kappa_2 \sqrt{y}
    \qquad\text{in } L^2(J),
  \end{equation*}
  hence
  \begin{equation*}
    \int_{M_\gamma(J)} |B|^2 \,d\mu
    \leq
    \liminf_{\eps \to 0}
    \int_{M_\eps(J)} u_\eps^2 |B_\eps|^2 \,d\mu_\eps
    \leq
    \liminf_{\eps \to 0}
    \int_{M_\eps} u_\eps^2 |B_\eps|^2 \,d\mu_\eps.
  \end{equation*}
  Since the right hand side is bounded independently of $J$, we obtain
  \begin{equation*}
    \int_{M_\gamma(\set{y>0} \sm S)} |B|^2 \,d\mu
    \leq
    \liminf_{\eps \to 0}
    \int_{M_\eps} u_\eps^2 |B_\eps|^2 \,d\mu_\eps
    <
    \infty
  \end{equation*}
  by exhausting $\set{y>0} \sm S$.
  The inequality $\calH^1(M_\gamma(S_u \cup S_\gamma)) < \infty$
  follows from \eqref{eq:proof:coerc-phase-fields} and the fact that
  each kink or interface $s \in S_u \cup S_\gamma$ carries at least an
  energy of $2\pi \sigma y(s)$ in the limit $\eps \to 0$. The details
  are given in the lower bound section and thus are omitted here.
\end{proof}

The following corollary renders the convergence around possible kinks
more precise and will be used to establish the lower bound.

\begin{corollary}
  \label{cor:proof:angle-coercivity}
  For any subsequence as in Lemma \ref{lem:proof:equi-coercivity}
  there are angle functions $\phi_\eps \in L^\infty(I) \cap
  W^{1,2}_{\loc}(I)$ of $\gamma_\eps$ that converge weakly in
  $BV_{\loc}(\set{y>0})$ to an angle function $\phi$ of $\gamma$ in
  $\set{y>0}$. Moreover, $\phi \in W^{1,2}(J \sm S)$ for any $J
  \Subset \set{y>0}$.
\end{corollary}

\begin{proof}
  Without loss of generality let $I = (0,\calL_\gamma)$ and
  $q_\gamma=1$. Since $\gamma_\eps \in W^{2,2}_{\loc}(I; \bbR^2)$ and
  $x_\eps' \geq 0$ by definition of $\calC$, there are angle functions
  $\phi_\eps \in W^{1,2}_{\loc}(I; [-\pi/2,\pi/2])$ of $\gamma_\eps$.
  Recalling $\phi_\eps' = -\kappa_{1,\eps} q_\eps$, uniform
  convergence of $y_\eps$, and Lemma \ref{lem:model:variation-bound},
  we fix $J \Subset \set{y>0}$ and obtain that
  \begin{equation*}
    \int_J |\phi_\eps'| \,d t
    =
    \int_J |\kappa_{1,\eps}| q_\eps \,d t
    \leq
    \frac{1}{2 \pi c_J} \int_{M_\eps} |B_\eps| \,d\mu_\eps
  \end{equation*}
  is uniformly bounded by Lemma~\ref{lem:model:variation-bound} for
  all sufficiently small $\eps>0$. Hence, $\phi_\eps$ is uniformly
  bounded in $W^{1,1}(J)$ and there exists a subsequence that
  converges weakly in $BV(J)$ to some $\phi$, that is, $\phi_\eps \to
  \phi$ in $L^1(J)$ and $\kappa_{1,\eps} \,d t$ restricted to $J$
  converges weakly to the measure $d\phi'$. Consequently,
  $\gamma_\eps' = q_\eps (\cos \phi_\eps, \sin \phi_\eps) \to (\cos
  \phi, \sin \phi) = \gamma'$ in $L^p(J; \bbR^2)$, that is, $\phi$ is
  an angle function of $\gamma$.
  Since this argument can be applied to any subsequence, convergence
  of the whole sequence $(\phi_\eps)$ in $BV(J)$ follows, and $\phi$
  is defined almost everywhere in $\set{y>0}$ by exhaustion.
  Arguing as for $\gamma_\eps''$ in Lemma
  \ref{lem:proof:equi-coercivity}, we obtain $\phi_\eps \wto \phi$ in
  $W^{1,2}(\widetilde J)$ for any $\widetilde J \Subset \set{y>0} \sm
  S$ and
  \begin{equation*}
    \int_{\widetilde J} |\phi'|^2 \,d t
    \leq
    \frac{4}{\pi c_{\widetilde J}} \liminf_{\eps \to 0}
    \int_{M_\eps} u_\eps^2 \kappa_{1,\eps}^2 \,d\mu_\eps
    <
    \infty.
  \end{equation*}
  Exhausting $J \Subset \set{y>0}$ by $\widetilde J \Subset J \sm S$
  yields $\phi' \in L^2(J \sm S)$ and $\phi \in W^{1,2}(J \sm S)$.
\end{proof}


\subsection{Lower bound}

To prove the lower bound
\begin{equation}
  \label{eq:proof:lb}
  \liminf_{\eps \to 0} \calF_\eps(\gamma_\eps,u_\eps)
  \geq
  \calF(\gamma,u)
\end{equation}
whenever $(\gamma_\eps,u_\eps)$ converges to $(\gamma,u)$ in $C^0(I;
\bbR^2) \times L^1(\set{y>0})$, it suffices to consider sequences such
that the left hand side of \eqref{eq:proof:lb} is finite and the limit
inferior is attained. Then by definition of $\calF_\eps$ we have
$(\gamma_\eps,u_\eps) \in \calC \times \calP$, and the equi-coercivity
result yields $(\gamma,u) \in \calD \times \calQ$ and the convergence
properties listed in Lemma \ref{lem:proof:equi-coercivity} and
Corollary \ref{cor:proof:angle-coercivity}.
In the following we consider the bulk energy $\calH$, kinks and
interfaces, and the axis of revolution separately.


\subsubsection{Bulk lower bound}

\begin{lemma}
  There holds
    \begin{equation*}
    \liminf_{\eps \to 0} \calH_\eps(\gamma_\eps,u_\eps,\set{y>0})
    \geq
    \calH(\gamma,u).
  \end{equation*}
\end{lemma}

\begin{proof}
  Let $J \Subset \set{y>0} \sm S$. From $\gamma_\eps \wto \gamma$ in
  $W^{2,2}(J;\bbR^2)$, $u_\eps \to u \in \set{\pm1}$ in $L^2(J)$, and
  $\sup_\eps \|u_\eps\|_{L^\infty(J)} < \infty$ we infer that
  \begin{equation*}
    u_\eps H_\eps \sqrt{|\gamma_\eps'| y_\eps} \wto u H \sqrt{|\gamma'| y}
    \qquad\text{in } L^2(J)
  \end{equation*}
  and, using \eqref{eq:surf:gauss-curv-angle} and $|\gamma_\eps'| =
  q_\eps$, that
  \begin{equation*}
    u_\eps^2 K_\eps |\gamma_\eps'| y_\eps
    =
    - u_\eps^2 \frac{y_\eps''}{q_\eps}
    \wto
    - u^2 \frac{y''}{q}
    =
    u^2 K |\gamma'| y
    \qquad \text{in } L^1(J).
  \end{equation*}
  Moreover, we have
  \begin{equation}
    \label{eq:proof:bulk-lb-convergence-added-part}
    u_\eps H_{s}(u_\eps) \sqrt{|\gamma_\eps'| y_\eps}
    \to u H_{s}(u) \sqrt{|\gamma'| y}
    \qquad\text{in } L^2(\set{y>0}).
  \end{equation}
  Hence the inequality
  \begin{equation}
    \label{eq:proof:bulk-lb}
    \calH(\gamma,u,J)
    +
    \int_{M_\gamma(J)} H_{s}(u)^2 \,d\mu
    \leq
    \liminf_{\eps \to 0}
    \left(
      \calH_\eps(\gamma_\eps,u_\eps,J)
      +
      \int_{M_\eps(J)} u_\eps^2 H_{s}(u_\eps)^2 \,d\mu_\eps
    \right)
  \end{equation}
  holds. As seen in \eqref{eq:model:bending-positivity}, the integrand
  on the right hand side of \eqref{eq:proof:bulk-lb} is non-negative,
  so we estimate the integral from above by extending its domain to
  $M_\eps(\set{y>0})$. The right hand side is then independent of $J
  \Subset \set{y>0} \sm S$, and by exhausting $\set{y>0} \sm S$ we
  obtain
  \begin{multline*}
    \calH(\gamma,u)
    + \int_{M_\gamma(\set{y>0} \sm S)}  H_{s}(u)^2 \,d\mu
    \\
    \leq
    \liminf_{\eps \to 0}
    \calH_\eps(\gamma_\eps,u_\eps,\set{y>0})
    +
    \limsup_{\eps \to 0}
    \int_{M_\eps(\set{y>0})} u_\eps^2 H_{s}(u_\eps)^2 \,d\mu_\eps.
  \end{multline*}
  The claim now follows from the convergence
  \eqref{eq:proof:bulk-lb-convergence-added-part}.
\end{proof}


\subsubsection{Kinks and interfaces}

Next we consider the interface energies $\calI_\eps$ and $\calI$ in
$\set{y>0}$. Points in $S \sm (S_u \cup S_\gamma)$ do not contribute
to the limit energy $\calI$, so it suffices to examine $s \in S_u \cup
S_\gamma$. In the following let $J \Subset \set{y>0}$ be an interval
around $s$ such that $\overline{J}\cap S = \set{s}$, which exists
because $S \cap \set{y > y(s)/2}$ is finite. Again, we assume
without loss of generality that $q_\gamma=1$.

If $s \in S_u \sm S_\gamma$ is an interface without kink, we estimate
the curvature term in $\calI_\eps$ from below by zero and use a
standard argument for the other terms as in \cite[Chapter
6]{Braides02}. That is, from the convergence of $u_\eps$ we deduce
that there are points $a_\eps,b_\eps \in J$ such that $a_\eps \to s$,
$b_\eps \to s$, $u_\eps(a_\eps) \to -1$, $u_\eps(b_\eps) \to +1$, and
without loss of generality $a_\eps < s < b_\eps$. By Young's
inequality and a change of variables we obtain
\begin{align*}
  \frac{1}{2\pi} \int_{M_\eps(a_\eps,b_\eps)}
  \eps |\snabla[M_\eps]{u_\eps}|^2 + \frac{1}{\eps} W(u_\eps) \,d\mu_\eps
  &\geq
  \left( \inf_{(a_\eps,b_\eps)} y_\eps \right)
  \int_{a_\eps}^{b_\eps} 2 \sqrt{W(u_\eps)}|u_\eps'| \,d t
  \\
  &\geq
  \left( \inf_{(a_\eps,b_\eps)} y_\eps \right)
  \int_{u_\eps(a_\eps)}^{u_\eps(b_\eps)} 2 \sqrt{W(u)} \,d u,
\end{align*}
and taking the lower limit as $\eps \to 0$ yields
\begin{equation}
  \label{eq:proof:interf-lb}
  \liminf_{\eps \to 0}
  \int_{M_\eps(a_\eps,b_\eps)}
  \eps |\snabla[M_\eps]{u_\eps}|^2 + \frac{1}{\eps} W(u_\eps) \,d\mu_\eps
  \geq
  2\pi y(s) \sigma.
\end{equation}

If $s \in S_u \cap S_\gamma$ is a kink and a proper interface, let
$(\phi_\eps)$ be angle functions of $(\gamma_\eps)$ in $J$ that converge
weakly in $BV(J)$ to an angle function $\phi$ of $\gamma$. We then
have
\begin{equation*}
  \int_J \phi_\eps' y_\eps \,d t
  \to
  [\phi](s)y(s)
  -
  \int_{J \sm \set{s}} \kappa_1 y \,d t,
\end{equation*}
where $\kappa_1 = -\phi' \in L^2(J\sm\set{s})$ is the curvature of
$\gamma$ in $J \sm \set{s}$ and $[\phi](s)$ the jump of the angle
$\phi$ at $s$.
Since $x_\eps' \geq 0$, we may assume $\phi_\eps \in [-\pi/2, \pi/2]$,
so that $[\gamma'] = [\phi] \in [-\pi,\pi]$.
The key step for the lower bound in $J$ is to formalise the intuition
that $\gamma_\eps$ approaches the kink where $u_\eps$ is close to
zero.

\begin{lemma}
  \label{lem:proof:kink-lb-1}
  For sufficiently small $\delta>0$ let $J_{\eps,\delta} = \set[t \in
  I]{|u_\eps| \leq \delta}$. Then
  \begin{equation*}
    \liminf_{\eps \to 0}
    \left|
      \int_{J \cap J_{\eps,\delta}} \phi_\eps' y_\eps \,d t
    \right|
    \geq
    y(s) |[\phi](s)|.
  \end{equation*}
\end{lemma}

\begin{proof}
  We show that the complement of $J_{\eps,\delta}$ in $J$ contains
  only the absolutely continuous part of $d\phi'$.
  Let $\beta>0$ be arbitrary but fixed, and let $U_\beta =
  [s-\beta,s+\beta]$. As $J \sm U_\beta \Subset \set{y>0} \sm S$, we
  have $|u_\eps| \geq 2\delta$ in $J \sm U_\beta$ for all sufficiently
  small $\eps$ according to Lemma \ref{lem:proof:equi-coercivity}, and
  therefore $J \cap J_{\eps,\delta} \subset U_\beta$.
  Writing $w_\eps = \phi_\eps' y_\eps + \kappa_1 y$, we have
  \begin{equation*}
    \left|
      \int_{J \sm J_{\eps,\delta}} w_\eps \,d t
    \right|
    \leq
    \left|
      \int_{J \sm U_\beta} w_\eps \,d t
    \right|
    +
    \int_{(J \sm J_{\eps,\delta}) \cap U_\beta} |w_\eps| \,d t
  \end{equation*}
  for all sufficiently small $\eps$. The first term on the right hand
  side converges to $0$ by weak convergence of $w_\eps$ in $J \sm
  U_\beta$, and the second integral is bounded by a constant times
  $\sqrt{\beta}$ due to H{\"o}lder's inequality and the uniform bound
  on the second fundamental forms of $M_{\eps}$ in $I \sm
  J_{\eps,\delta}$.
  As $\beta>0$ is arbitrary, we obtain
  \begin{equation*}
    \limsup_{\epsilon \to 0}
    \left|
      \int_{J \sm J_{\eps,\delta}} w_\eps \,d t
    \right|
    = 0,
  \end{equation*}
  and taking the lower limit in the inequality
  \begin{equation*}
    \left|
      \int_{J \cap J_{\eps,\delta}} w_\eps \,d t
    \right|
    \geq
    \left|
      \int_{J} w_\eps \,d t
    \right|
    -
    \left|
      \int_{J \sm J_{\eps,\delta}} w_\eps \,d t
    \right|
  \end{equation*}
  yields the claim because $|J \cap J_{\eps,\delta}| \to 0$ as
  $\eps\to0$ due to the uniform bound on
  \begin{equation*}
    \int_{M_\eps(J \cap J_{\eps,\delta})}
    \frac{1}{\eps} W(u_\eps) \,d\mu_\eps
    \geq
    \left( \inf_{[-\delta,\delta]} W \right)
    \frac{|J \cap J_{\eps,\delta}|}{\eps}
  \end{equation*}
  and $y\kappa_1 \in L^2(J \sm \set{s})$.
\end{proof}

Using the above splitting of $J$ into $J \cap J_{\eps,\delta}$ and $J
\sm J_{\eps,\delta}$, we prove the lower bound inequality.

\begin{lemma}
  \label{lem:proof:kink-lb-2}
  There holds
  \begin{equation*}
    \liminf_{\eps \to 0}
    \calI_\eps(\gamma_\eps,u_\eps,J)
    \geq
    2 \pi
    \left(
      \hat\sigma |[\gamma'](s)| + \sigma
    \right)
    y(s).
  \end{equation*}
\end{lemma}

\begin{proof}
  With the notation of the Lemma \ref{lem:proof:kink-lb-1} we
  have
  \begin{equation*}
    \frac{\calI_\eps(\gamma_\eps,u_\eps,J)}{2\pi}
    \geq
    \int_{J \cap J_{\eps,\delta}}
    \left(
      \frac{\eps}{q_\eps}|\phi_\eps'|^2
      + \frac{q_\eps}{\eps} W(u_\eps) 
    \right) y_\eps \,d t 
    +
    \int_{J \sm J_{\eps,\delta}}
    \left( 
      \frac{\eps}{q_\eps}|u_\eps'|^2 + \frac{q_\eps}{\eps} W(u_\eps) 
    \right) y_\eps \,d t.
  \end{equation*}
  Estimating the first term on the right hand side with Young's
  inequality we obtain
  \begin{align*}
    \int_{J \cap J_{\eps,\delta}}
    \left(
      \frac{\eps}{q_\eps}|\phi_\eps'|^2
      + \frac{q_\eps}{\eps} W(u_\eps) 
    \right) y_\eps \,d t 
    &\geq
    \int_{J \cap J_{\eps,\delta}}
    2 \sqrt{W(u_\eps)} |\phi_\eps'| y_\eps \,d t
    \\
    &\geq
    2 \inf_{u \in [-\delta,\delta]} \sqrt{W(u)}
    \left|
      \int_{J \cap J_{\eps,\delta}} \phi_\eps' y_\eps \,d t
    \right|.
  \end{align*}
  With the second integral we deal as in
  \eqref{eq:proof:interf-lb}; the only difference is that we
  now find an interval $(a_\eps,b_\eps) \subset J \sm J_{\eps,\delta}$
  such that $u_\eps(a_\eps) \to \delta$, $u_\eps(b_\eps) \to 1$ on one
  side of $s$, and the same with $-\delta$ and $-1$ on the other side.
  Combining both estimates and taking the lower limit as $\eps \to 0$
  yields
  \begin{align*}
    \frac{1}{2\pi}
    \liminf_{\eps \to 0}
    \calI_\eps(\gamma_\eps,u_\eps,J)
    &\geq
    2 y(s) |[\phi](s)| \inf_{[-\delta,\delta]} \sqrt{W(u)}
    \\
    &\quad+
    2 y(s) \int_{\delta}^{1} \sqrt{W(u)} \,d u
    +
    2 y(s) \int_{-1}^{-\delta} \sqrt{W(u)} \,d u.
  \end{align*}
  Taking the supremum over $\delta > 0$ finishes the proof.
\end{proof}

Finally, if $s \in S_\gamma \sm S_u$ is a ghost interface, then the
phase field $u$ is constant in $\overline{J}$, say $u \equiv 1$. The
argument with the splitting of $J$ into $J \cap J_{\eps,\delta}$ and
$J \sm J_{\eps,\delta}$ is as in Lemma \ref{lem:proof:kink-lb-2}, but
now there is an interval $(a_\eps,b_\eps) \subset J \sm
J_{\eps,\delta}$ such that $u_\eps(a_\eps) \to \delta$,
$u_\eps(b_\eps) \to 1$ on either side of $s$. Hence, we conclude
\begin{equation*}
  \frac{1}{2\pi}
  \liminf_{\eps \to 0}
  \calI_\eps(\gamma_\eps,u_\eps,J)
  \geq
  \hat\sigma |[\gamma'](s)| y(s)
  +
  4 y(s) \int_0^1 \sqrt{W(u)} \,d u,
\end{equation*}
and the right hand side is equal to $\hat\sigma |[\gamma'](s)| y(s) +
\sigma y(s)$ due to the symmetry of $W$. The same argument holds when
$u \equiv -1$ near $s$.

The above reasoning extends to any finite subset $S$ of $S_\gamma \cup
S_u$, and we obtain
\begin{equation*}
  \liminf_{\eps \to 0}
  \calI_\eps(\gamma_\eps,u_\eps,\set{y>0})
  \geq
  2\pi \sum_{s \in S}
  \left( \hat\sigma |[\gamma'](s)| + \sigma \right) y(s).
\end{equation*}
Since the left hand side is independent of $S$, we conclude the lower
bound for kinks and interfaces
\begin{equation*}
  \liminf_{\eps \to 0}
  \calI_\eps(\gamma_\eps,u_\eps, \set{y>0})
  \geq
  2\pi \sum_{s \in S_\gamma \cup S_u} \left(
    \sigma + \hat\sigma |[\gamma'](s)| \right) y(s)
  =
  \calI(\gamma,u, \set{y>0}).
\end{equation*}


\subsubsection{Axis of revolution}

To motivate the lower bound estimate at the axis of revolution, we
first consider the simple example that $\gamma_\eps(t)=(q_\eps
t,y_\eps)$, $y_\eps \in \bbR$ is a straight horizontal line segment in
$R = \set{y=0}$ such that $y_\eps \to 0$ as $\eps \to 0$. Then
$\kappa_{1,\eps}=0$, while $\kappa_{2,\eps} = 1 / y_\eps$ blows up and
contributes to the limit of $\calF_\eps$. From the uniform bound on
\begin{equation*}
  \int_{M_\eps(R)} u_\eps^2 \kappa_{2,\eps}^2 \,d\mu_\eps
  \geq
  \frac{2\pi}{q_\eps \sup_R y_\eps} \int_R u_\eps^2 x_\eps'^2 \,d t
\end{equation*}
and $x_\eps' = q_\eps \to q = x'$ in $R$ we get $u_\eps \to 0$ in
$R$. Therefore, the potential term contributes to the limit energy.
On the other hand, there is no reason for $u_\eps$ to have a large
gradient in $R$, and it is reasonable to assume that $u_\eps$
tends to zero sufficiently fast so that there is no contribution of
$\calH_\eps(\gamma_\eps, u_\eps, R)$ in the limit.
Then
\begin{equation*}
  \calF_\eps(\gamma_\eps,u_\eps,R)
  \sim
  \int_{M_\eps(R)}
  \frac{1}{\eps} W(u_\eps) + \eps \kappa_{2,\eps}^2 \,d\mu_\eps
  \sim
  \int_{R} \frac{y_\eps}{\eps} + \frac{\eps}{y_\eps} \,d t
\end{equation*}
is bounded as $\eps\to0$ if and only if $y_\eps \sim \eps$, and in
this case $\calF_\eps(\gamma_\eps,u_\eps,R) \sim \calH^1(R)$.

To extend this reasoning to general $(\gamma_\eps, u_\eps)$, when in
particular the behaviour of $u_\eps$ is not known, recall that
$\calA_{\eps}(R) \to 0$ as $\eps \to 0$ and $\|u_\eps\|_\infty
\leq C_0$. These properties imply
\begin{equation*}
  \int_{M_\eps(R)} u_\eps^2 H_{s}(u_\eps)^2 \,d\mu_\eps
  \to 0
  \qquad\text{as}\qquad \eps \to 0,
\end{equation*}
and with \eqref{eq:model:bending-positivity} we conclude
\begin{equation*}
  \liminf_{\eps \to 0}
  \calH_\eps(\gamma_\eps,u_\eps,R)
  =
  \liminf_{\eps \to 0}
  \left(
    \calH_\eps(\gamma_\eps,u_\eps,R)
    +
    \int_{M_\eps(R)} u_\eps^2 H_{s}(u_\eps)^2 \,d\mu_\eps
  \right)
  \geq
  0.
\end{equation*}
For the interface energy we consider again $J_{\eps,\delta} =
\set{|u_\eps| \leq \delta}$. Similar to the proof of Lemma
\ref{lem:proof:kink-lb-1}, H{\"o}lder's inequality yields
\begin{equation}
  \label{eq:proof:aor-hoelder}
  \left(
    \delta \int_{M_\eps(R \sm J_{\eps,\delta})} |B_\eps| \,d\mu_\eps
  \right)^2
  \leq
  \calA_\eps(R)
  \int_{M_\eps} u_\eps^2 |B_\eps|^2 \,d\mu_\eps,
\end{equation}
and the right hand side of \eqref{eq:proof:aor-hoelder} vanishes in
the limit $\eps\to0$. Thus, using Young's inequality and $x_\eps'
\wsto x'$ in $L^\infty(I)$, we find
\begin{align*}
  \liminf_{\eps \to 0}
  \calI_\eps(\gamma_\eps,u_\eps,R)
  &\geq
  \liminf_{\eps \to 0}
  \int_{M_\eps(R \cap J_{\eps,\delta})}
  \eps |B_\eps|^2 + \frac{1}{\eps} W(u_\eps) \,d\mu_\eps
  \\
  &\geq
  2 \left( \inf_{u \in [-\delta,\delta]} \sqrt{W(u)} \right)
  \liminf_{\eps \to 0}
  \int_{M_\eps(R \cap J_{\eps,\delta})} |B_\eps| \,d\mu_\eps
  \\
  &=
  2 \left( \inf_{u \in [-\delta,\delta]} \sqrt{W(u)} \right)
  \liminf_{\eps \to 0}
  \int_{M_\eps(R)} |B_\eps| \,d\mu_\eps
  \\  
  &\geq
  4\pi \left( \inf_{u \in [-\delta,\delta]} \sqrt{W(u)} \right)
  \int_R x' \,d t.
\end{align*}
Taking the supremum over $\delta>0$ and combining with the
estimate for $\calH_\eps$ yields
\begin{equation}
  \label{eq:proof:aor-lb}
  \liminf_{\eps \to 0}
  \calF_\eps(\gamma_\eps,u_\eps,R)
  \geq
  \hat\sigma
  \liminf_{\eps \to 0}
  \int_{M_\eps(R)} |B_\eps| \,d\mu_\eps
  \geq
  2 \pi \hat\sigma \int_R x' \,d t.
\end{equation}
Due to $x' \geq 0$ and $y'=0$ in $R$, we have
\begin{equation*}
  2\pi \int_R x' \,d t
  =
  2\pi \int_R |\gamma'| \,d t
  =
  2\pi \calL_\gamma(R)
  =
  \calH^1(M_\gamma(R)),
\end{equation*}
and this concludes the proof of the lower bound
\eqref{eq:proof:lb}.


\subsection{Upper bound inequality}
\label{sec:proof:ub}

This section is devoted to the upper bound inequality
\begin{equation*}
  \limsup_{\eps \to 0} \calF_\eps(\gamma_\eps, u_\eps)
  \leq
  \calF(\gamma,u)
\end{equation*}
whenever $(\gamma,u) \in \calD \times \calQ$ has finite energy and
$\gamma$ is parametrised with constant speed.
We first approximate $(\gamma,u)$ by a sequence of simple membranes in
$\calD \times \calQ$ that have a finite number of components and
finitely many (ghost) interfaces. We construct a recovery sequence for
such a simple membrane by employing essentially local changes to curve
and phase field. A diagonal sequence then recovers $(\gamma,u)$.

Throughout this section we assume that $(\gamma,u) \in \calD \times
\calQ$ has finite energy $\calF(\gamma,u)$ and constant speed
$|\gamma'| \equiv q$ in $I$, without loss of generality $q=1$.
Since the value of $u$ in $\set{y=0}$ does not enter the energy or our
arguments, we assume $u=0$ in $\set{y=0}$.


\subsubsection{Approximation by simple configurations}

\begin{lemma}
  \label{lem:proof:approx-noc-finite}
  Assume that $M_\gamma$ has infinitely many components.
  Then there is a sequence $(\gamma_\delta,u_\delta) \in \calD \times
  \calQ$ that converges to $(\gamma,u)$ in $C^0(I; \bbR^2) \times
  L^1(\set{y>0})$ as $\delta \to 0$ such that
  $\calF(\gamma_\delta,u_\delta) \to \calF(\gamma,u)$ and each
  $M_{\delta}$ has finitely many components.
\end{lemma}

\begin{proof}
  Since $\calL_\gamma$, $\calA_\gamma$, and $|\calF(\gamma,u)|$ are
  finite, approximations $(\gamma_\delta,u_\delta)$ can be constructed
  by replacing all components of $(\gamma,u)$, whose curve length is
  less than $\delta$, with a horizontal on the axis of revolution and
  phase field equal to zero.
  Convergence of curves and phase fields as $\delta\to0$ are easily
  checked. The energy difference consists of the total energy of the
  removed components and the interface energy of the new horizontals
  on the axis of revolution. Thus we have
  \begin{equation*}
    |\calF(\gamma,u)-\calF(\gamma_\delta,u_\delta)|
    \leq
    \sum_{\omega: \calL_\gamma(\omega) \leq \delta}
    \Big(
      | \calF(\gamma,u,\omega)| + 2\pi \hat\sigma \calL_\gamma(\omega)
    \Big),
  \end{equation*}
  and both terms on the right hand side converge to $0$ as $\delta \to
  0$. The area difference satisfies
  \begin{equation*}
    \calA_\gamma - \calA_{\gamma_\delta}
    =
    2\pi \sum_{\omega: \calL_\gamma(\omega)\leq\delta}
    \int_\omega |\gamma'| y \,d t
    \leq
    2\pi \delta \sum_{\omega: \calL_\gamma(\omega)\leq\delta}
    \calL_\gamma(\omega)
    =
    o(\delta),
  \end{equation*}
  and it remains to recover the constraints exactly so that
  $(\gamma_\delta,u_\delta)$ is admissible.

  First, if there is an interval $J \Subset \set{y>0} \sm S_\gamma
  \cup S_u$ such that $x'>0$ in $J$, then the corresponding component
  belongs to $(\gamma_\delta,u_\delta)$ for all sufficiently small
  $\delta>0$ and we can add a perturbation that is compactly supported
  in $J$, tends to zero in $W^{2,2}(J;\bbR^2)$ as $\delta \to 0$, and
  recovers the area; if necessary, a reparametrisation fixes the
  constant speed requirement.
  If there is no such interval $J$, then $\gamma$ consists only of
  vertical line segments interrupted by kinks and the area constraint
  is easily established for $M_{\gamma_\delta}$ by adapting the length
  of two adjacent line segments.

  Second, if there is at least one proper interface without kink in
  $(\gamma,u)$ then this interface also belongs to
  $(\gamma_\delta,u_\delta)$ for all sufficiently small $\delta$. It
  can be moved by an order less than $\delta$ and with change in
  energy of the same order to recover the phase integral constraint.
  If $(\gamma,u)$ contains no proper interface, introducing one or a
  finite number of new interfaces at a height of order less than
  $\sqrt{\delta}$ above the axis of revolution and flipping the sign
  of $u_\delta$ below these new interfaces recovers the
  constraint. The change in energy contributed by each new interface
  is proportional to its height above the axis of revolution and thus
  vanishes in the limit $\delta\to0$.
\end{proof}

\begin{lemma}
  \label{lem:proof:approx-noi-finite}
  Assume that $M_\gamma$ has finitely many components.
  Then there is a sequence $(\gamma_\delta,u_\delta) \in \calD \times
  \calQ$ that converges to $(\gamma,u)$ in $C^0(I; \bbR^2) \times
  L^1(\set{y>0})$ as $\delta \to 0$ such that
  $\calF(\gamma_\delta,u_\delta) \to \calF(\gamma,u)$ and
  $\calH^0(S_{\gamma_\delta} \cup S_{u_\delta}) < \infty$.
  Every component $\omega = (a,b)$ of $M_{\gamma_\delta}$ meets the
  axis of revolution in a line perpendicular to it, that is
  $\gamma_\delta' = (0,1)$ near $a$ in $\omega$ and $\gamma_\delta' =
  (0,-1)$ near $b$.
  Two adjacent components are connected by a horizontal segment on the
  axis of revolution.
\end{lemma}

\begin{proof}
  The approximations are constructed by changing $(\gamma,u)$ in
  segments of length of order $\delta$ around component
  boundaries.
  More precisely, let $\omega = (a,b)$ be a component of $M_\gamma$
  and $\delta>0$ sufficiently small so that $b_\delta = b - \delta \in
  \omega$. In $(b_\delta,b]$ we replace $\gamma$ by 
  \begin{equation*}
    \gamma_\delta(t) =
    \begin{cases}
      ( x(b_\delta), y(b_\delta) - t + b_\delta )
      &\text{if } t \in (b_\delta, \hat b_\delta),
      \\
      ( x(b_\delta) + t - \hat b_\delta, 0 )
      &\text{if } t \in (\hat b_\delta, b),
    \end{cases}
  \end{equation*}
  that is, we move vertically down until we reach the axis of
  revolution at $\hat b_\delta = y(b_\delta) + b_\delta$ and fill the
  remaining interval $(\hat b_\delta,b]$ by moving to the right.
  At the other component boundary $a$ we do the same but with the
  horizontal to the left.

  Making this replacement for every component, shifting remaining
  segments of $\gamma$ slightly in $x$-direction to glue all parts
  together continuously and setting the phase field to, say, $+1$ on
  the new verticals and $0$ on the horizontals we obtain
  $(\gamma_\delta,u_\delta)$ such that $\gamma_\delta \to \gamma$ in
  $C^0(I;\bbR^2)$ and $u_\delta \to u$ in $L^1(I)$.
  Denoting by $M_{orig}$ all parts of $M_\gamma$ that have been
  removed and by $M_{hor}$ and $M_{ver}$ the introduced horizontals
  and verticals, the Helfrich energy difference is bounded by
  \begin{align*}
    |\calH(\gamma,u)-\calH(\gamma_\delta,u_\delta)|
    &\leq
    \int_{M_{orig}} |H_\gamma-H_{s}(u)|^2 + |K_\gamma| \,d\mu
    +
    \int_{M_{ver}} H_{s}(u_\delta)^2 \,d\mu_\delta.
  \end{align*}
  The second term is bounded by $\|H_{s}\|_\infty^2
  \mu_\delta(M_{ver}) \to 0$ as $\delta \to 0$, and the first tends to
  $0$ as $\delta \to 0$ due to $\mu(M_{orig}) \to 0$ and uniform
  continuity of the integral.
  The difference in the interface energy consists of original (ghost)
  interfaces that are omitted in $(\gamma_\delta, u_\delta)$, the two
  probably introduced kinks for each component, and the new pieces on
  the axis of revolution. Therefore, we obtain
  \begin{align*}
    \frac{1}{2\pi}
    |\calI(\gamma,u)-\calI(\gamma_\delta,u_\delta)|
    &\leq
    \sum_{\substack{s \in S_\gamma \cup S_u\\y(s) \leq \delta}}
    (\sigma + \hat \sigma |[\gamma'](s)|)y(s)
    +
    2 N_\gamma (\sigma + \hat\sigma\pi) \delta
    +
    2 N_\gamma \hat\sigma \delta,
  \end{align*}
  where $N_\gamma$ denotes the number of components of $\gamma$. The
  first term converges to $0$ as $\delta \to 0$ because the sum over
  all (ghost) interfaces is finite, thus the energy difference
  vanishes in the limit $\delta \to 0$.
  Since $y_\delta \leq y \leq \delta$ where $\gamma$ has been
  replaced, one easily finds that $|\calA_\gamma - \calA_{\delta}|$ is
  of order $\delta^2$; thus the constraints can be recovered as in
  Lemma~\ref{lem:proof:approx-noc-finite}.

  If necessary, additional minor changes such as adding horizontal
  segments between adjacent components or removing horizontals at
  $\partial I$ can be applied.
\end{proof}

From now on we assume that $(\gamma,u)$ has the form of the
approximations constructed in Lemma \ref{lem:proof:approx-noi-finite}.
As an example, Figure~\ref{fig:proof:curve-d-simplified} shows an
approximation of the curve in Figure~\ref{fig:model:curve-in-d}.

\begin{figure}
  \centering
  \includegraphics[width=.9\textwidth]{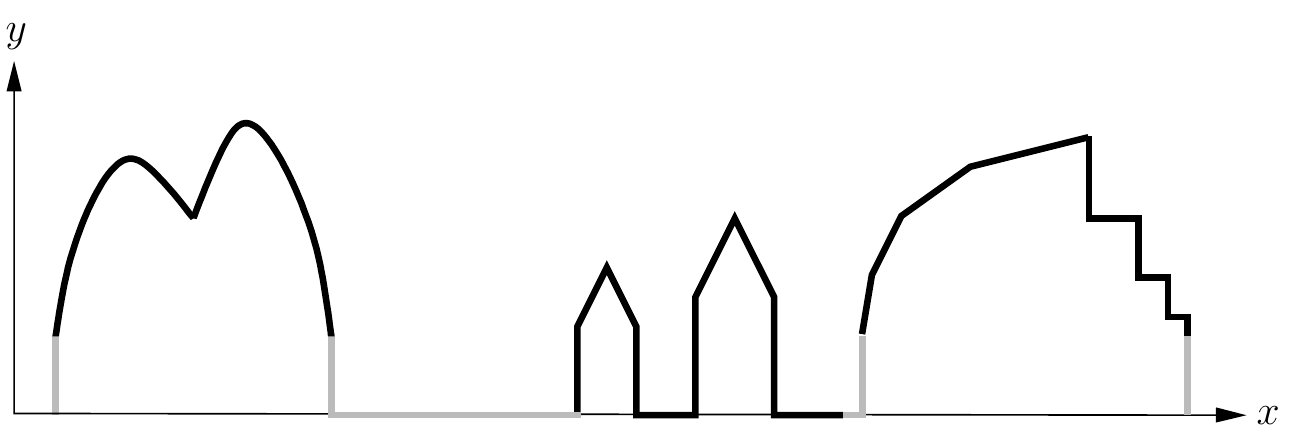}
  \caption{Approximation of the curve in Figure
    \ref{fig:model:curve-in-d}, original segments in black, new parts
    in grey. Small components are removed and the ends of the
    remaining components are replaced by line segments; between two
    adjacent components there is always a small horizontal segment on
    the axis of revolution. The number of (ghost) interfaces is
    finite.}
  \label{fig:proof:curve-d-simplified}
\end{figure}


\subsubsection{Kinks and interfaces}
\label{sec:kink:k-and-i-ub}

Let $s \in S_{\gamma}$ be a kink, $S = S_\gamma \cup S_u$, and $J
\Subset \set{y>0}$ with $\overline J \cap S = \set{s}$. For simplicity
of notation we formulate the following arguments for curves and phase
fields given in an interval $J$ around $s=0$; recall that
$|\gamma'|=1$ in $J$.
First we smooth out kinks by a linear interpolation of the tangent
angle of $\gamma$ around $s=0$. This local procedure disturbs the
constraints and disrupts the curve, so that we have to add some
corrections, one of which is a global shift in $x$-direction
because of the requirement $x' \geq 0$.

\begin{lemma}
  \label{lem:proof:kink-recovery}
  Let $J=(-a,a)$.
  For all sufficiently small $\eps>0$ there is $\gamma_\eps =
  (x_\eps,y_\eps) \in W^{2,2}(J; \bbR^2)$ such that 
  \begin{itemize}
  \item $\gamma_\eps$ satisfies $\inf_J y_\eps >0$ and $x_\eps' \geq
    0$;
  \item $\gamma_\eps$ fits almost into $\gamma$, that is, at the end
    points of $\gamma_\eps(J)$ we have
    \begin{equation*}
      \gamma_\eps(-a) = \gamma(-a),
      \quad
      \gamma_\eps'(-a) = \gamma'(-a),
      \quad
      \gamma_\eps(a) = (x(a) + o(1), y(a)),
      \quad
      \gamma_\eps'(a) = \gamma'(a);
    \end{equation*}
  \item $\gamma_\eps \to \gamma$ in $W^{1,p}(J; \bbR^2)$ for any $p \in
    [1,\infty)$ as $\eps \to 0$;
  \item $\calA_{\eps}(J) = \calA_\gamma(J) + O(\eps)$ and
    $\displaystyle \int_{M_\eps(J)} u \,d\mu_\eps = \int_{M_\gamma(J)}
    u \,d\mu + O(\eps)$ as $\eps \to 0$; and
  \item with $J_\eps=(-\delta_\eps,\delta_\eps)$, where $\delta_\eps =
    \frac{|[\gamma']|}{\hat\sigma} \eps$, there holds
    \begin{equation*}
      \lim_{\eps \to 0} \calI_\eps(\gamma_\eps,0,J_\eps)
      =
      2 \pi \hat\sigma |[\gamma'](0)| y(0).
    \end{equation*}
  \end{itemize}
  Moreover, $\gamma_\eps' = \gamma' + r_\eps$ in $J \sm J_\eps$ where
  $\spt r_\eps \Subset J \sm J_\eps$ is independent of $\eps$ and
  $r_\eps \to 0$ in $W^{1,2}(J;\bbR^2)$ as $\eps \to 0$.
\end{lemma}

\begin{proof}
  Let $\phi$ be an angle function for $\gamma$ in $J$ that is
  uniformly continuous on either side of $s=0$ and satisfies $|\phi|
  \leq \pi/2$. Denote by $\phi^+$ and $\phi^-$ the one-sided limit of
  $\phi$ at $s=0$ from the right and the left, respectively; then the
  kink carries the ``bending energy'' $2\pi \hat\sigma y(s) |\phi^+ -
  \phi^-|$ and we have $\delta_\eps = |\phi^+-\phi^-|\eps /
  \hat\sigma$.

  In the simple case that $\gamma$ consists of two straight lines in
  $J$, the linear interpolation $\phi_\eps \in W^{1,p}(J)$ of
  $\phi^\pm$ in $J_\eps = (-\delta_\eps,\delta_\eps) \subset J$ is
  given by
  \begin{equation*}
    \phi_\eps(t) =
    \begin{dcases}
      \phi^-
      &\text{if } t < - \delta_\eps,
      \\
      \frac{\phi^+ - \phi^-}{2\delta_\eps} t + \frac{\phi^+ + \phi^-}{2}
      &\text{if } -\delta_\eps \leq t < \delta_\eps,
      \\
      \phi^+
      &\text{if } \delta_\eps \leq t.
    \end{dcases}
  \end{equation*}
  The curve $\gamma_\eps$, defined by $\gamma_\eps' = (\cos\phi_\eps,
  \sin\phi_\eps)$ and $\gamma_\eps(-\delta_\eps) =
  \gamma(-\delta_\eps)$, converges in $W^{1,p}(J;\bbR^2)$ to $\gamma$
  as $\eps \to 0$. Using Young's inequality one easily verifies
  \begin{align}
    \label{eq:proof:two-lines-energy}
    \frac{1}{2\pi}
    \calI_\eps(\gamma_\eps,0,J_\eps)
    &=
    \left(
      \frac{2\delta_\eps}{\eps} W(0) +
      \frac{\eps}{2\delta_\eps} (\phi^+ - \phi^-)^2
    \right)
    \frac{1}{2 \delta_\eps} \int_{-\delta_\eps}^{\delta_\eps} y_\eps \,d t
    +
    \eps \int_{-\delta_\eps}^{\delta_\eps} \kappa_{2,\eps}^2 y_\eps \,d t
    \nonumber
    \\
    &\geq
    \hat\sigma |\phi^+ - \phi^-|
    \frac{1}{2 \delta_\eps} \int_{-\delta_\eps}^{\delta_\eps} y_\eps \,d t
    +
    \eps \int_{-\delta_\eps}^{\delta_\eps} \kappa_{2,\eps}^2 y_\eps \,d t,
  \end{align}
  and by our choice of $\delta_\eps$ we have equality in
  \eqref{eq:proof:two-lines-energy}.
  The first integral in~\eqref{eq:proof:two-lines-energy} divided by
  $2\delta_\eps$ converges to $y(0)$ as $\eps \to 0$, and the second
  term vanishes because the integral of $\kappa_{2,\eps}^2 y_\eps$ is
  bounded. Thus $\calI_\eps(\gamma_\eps,0,J_\eps) \to 2\pi \hat\sigma
  |[\gamma']| y(0)$ as desired.

  For a general angle function $\phi$ the interpolation is
  \begin{equation*}
    \phi_\eps(t) =
    \begin{dcases}
      \phi(t)
      &\text{if }  |t| > \delta_\eps, \\
      \frac{\left( \phi(\delta_\eps)-\phi(-\delta_\eps)
        \right)}{2\delta_\eps} t +
      \frac{\left( \phi(\delta_\eps)+\phi(-\delta_\eps)
        \right)}{2}
      &\text{if } |t| \leq \delta_\eps,
    \end{dcases}
  \end{equation*}
  see Figure~\ref{fig:proof:kink-recovery}, and similarly as above we
  get
  \begin{align*}
    \frac{1}{2\pi}
    \calI_\eps(\gamma_\eps,0,J_\eps)
    &=
    \sqrt{W(0)}
    \left(
      |[\phi]| + \frac{|\phi(\delta_\eps)-\phi(-\delta_\eps)|^2}{|[\phi]|}
    \right)
    \frac{1}{2\delta_\eps} \int_{-\delta_\eps}^{\delta_\eps} y_\eps \,d t
    +
    \eps \int_{-\delta_\eps}^{\delta_\eps} \kappa_{2,\eps}^2 y_\eps \,d t
    \\
    &\to
    \hat\sigma |[\gamma']| y(0).
  \end{align*}
  By construction, $\phi_\eps \in [-\pi/2,\pi/2]$, that is $x_\eps'
  \geq 0$, and $|\gamma_\eps'| \equiv 1$ in $J$. Also, $\gamma_\eps
  \to \gamma$ in $W^{1,p}(J; \bbR^2)$ because $\phi_\eps \to \phi$ in
  $L^p(J)$ for any $p \in [1,\infty)$. Therefore, $y_\eps \geq \inf_J
  y/2 > 0$ for all sufficiently small $\eps$.

  \begin{figure}
    \centering
    \includegraphics[width=.4\textwidth]{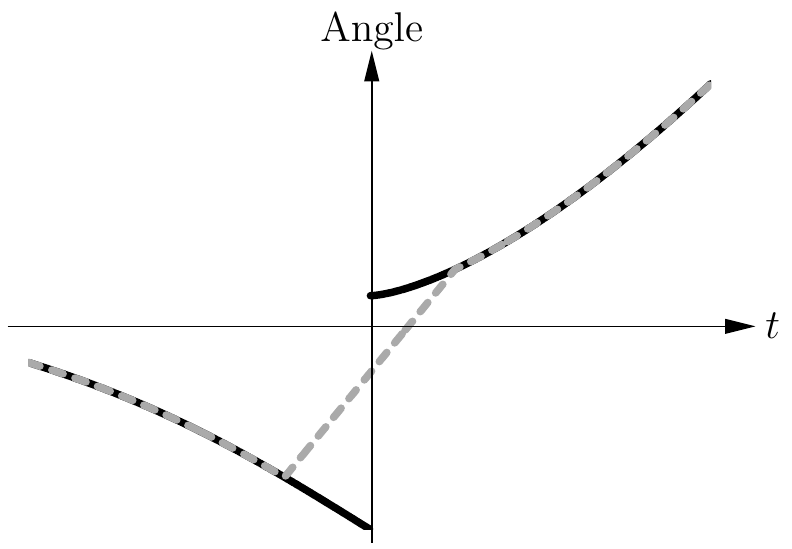}
    \hspace*{.05\textwidth}
    \includegraphics[width=.4\textwidth]{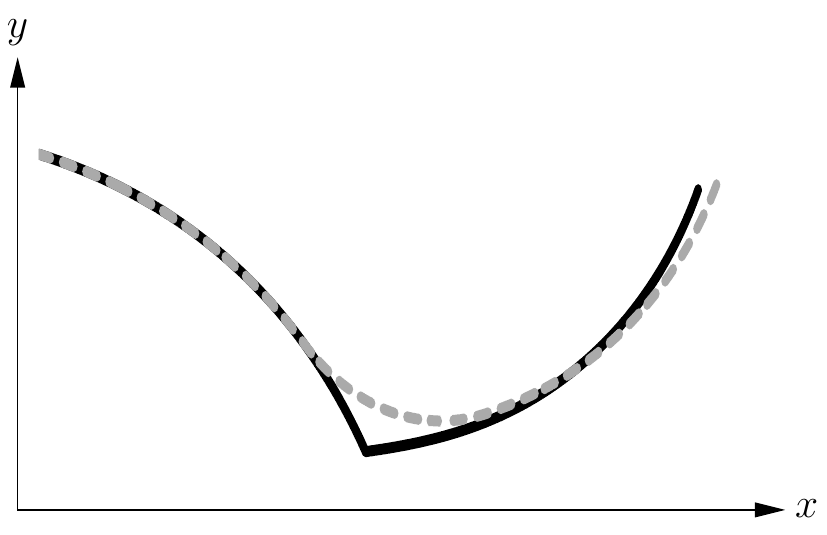}
    \caption{Linear interpolation of the tangent angle on the left and
      the corresponding curve on the right. Original angle and curve
      are black, interpolations grey.}
    \label{fig:proof:kink-recovery}
  \end{figure}

  It remains to correct the $y$-coordinate of the right end point of
  $\gamma_\eps(J)$ and to calculate the error in the area and phase
  integral constraint.
  We fix $\widetilde J \Subset J \sm J_\eps$ independently of all
  sufficiently small $\eps$ and $f \in C_c^\infty(\widetilde J)$ such
  that $f\geq0$ and $\int_{\widetilde J} f \,d t = 1$. The perturbed
  curve $\widetilde \gamma_\eps = (x_\eps, y_\eps + \alpha_\eps F)$,
  where $F(t) = \int_{-a}^t f(s) \,d s$, has the desired end point
  $y$-coordinate for
  \begin{equation*}
    \alpha_\eps = y(a) - y_\eps(a).
  \end{equation*}
  Since
  \begin{equation*}
    |\gamma_\eps(t) - \gamma(t)|
    \leq
    \int_{-\delta_\eps}^{\delta_\eps} |\gamma_\eps' - \gamma'| \,d s
    \leq
    4 \delta_\eps
    =
    O(\eps),
  \end{equation*}
  also $\alpha_\eps$, $\|\widetilde \gamma_\eps - \gamma\|_\infty$, 
  and $\| \widetilde \gamma_\eps' - \gamma' \|_{L^\infty(J \sm
    J_\eps)}$ are at most of order $\eps$. The claims for area and
  phase constraint follow, and $r_\eps = (0, \alpha_\eps f)$.
\end{proof}

Next we construct a recovery sequence for the phase field in $J$ which
is in line with $\gamma_\eps$ of Lemma~\ref{lem:proof:kink-recovery}.
It is well known, see for instance~\cite{Alberti00}, that in the
classical one-dimensional Modica-Mortola setting the optimal
$\eps$-energy profile for a transition of $u_\eps$ from $-1$ to $+1$
is obtained by minimising
\begin{equation*}
  G_\eps(u) = \int_{\bbR} \eps |u'|^2 + \frac{1}{\eps} W(u) \,d t
\end{equation*}
among functions $u$ that satisfy $u(0)=0$ and $u(\pm\infty) = \pm
1$. Indeed, setting $u_\eps(t) = u(t/\eps)$, we find
\begin{equation*}
  G_\eps(u_\eps) = G_1(u)
  \geq
  2\int_{\bbR} \sqrt{W(u)}u' \,d t
  =
  2 \int_{\bbR} \sqrt{W(u)} \,d u
  =
  \sigma.
\end{equation*}
Equality holds if and only if
\begin{equation}
  \label{eq:proof:phase-field-equation}
  u'=\sqrt{W(u)}, 
\end{equation}
which admits a local solution $p$ with initial condition $p(0)=0$
because $\sqrt{W(\cdot)}$ is continuous. Obviously, the constants $+1$
and $-1$ are a global super- and sub-solution of
\eqref{eq:proof:phase-field-equation}, hence $p$ can be extended to
the whole real line.  Since $W(p)>0$ for $p\in(-1,+1)$, $p(t)$
converges to $\pm1$ as $t\to\pm\infty$. Thus $p(t/\eps)$ minimises
$G_\eps$, and due to the symmetry of $W$ we may assume $-p(-t)=p(t)$.

The building block $p_\eps$ of our phase field recovery is given by
\begin{equation*}
  p_\eps(t) =
  \begin{cases}
    0
    &\text{if } 0 \leq t \leq \delta_\eps, \\
    p\left( \frac{t-\delta_\eps}{\eps} \right)
    &\text{if } \delta_\eps < t \leq \delta_\eps+\sqrt{\eps}, \\
    p(1/\sqrt{\eps}) + \tfrac{1}{\eps} \left(
      t-\delta_\eps-\sqrt{\eps} \right)
    &\text{if } \delta_\eps+\sqrt{\eps} < t \leq \delta_\eps+\sqrt{\eps}
    + \eps \left(
      1-p(1/\sqrt{\eps}) \right), \\
    1
    &\text{if } \delta_\eps+\sqrt{\eps} + \eps \left(
      1-p(1/\sqrt{\eps}) \right) < t,
  \end{cases}
\end{equation*}
which connects $p_\eps=0$ and $p_\eps=1$ by an appropriately scaled
optimal profile and a linear segment. In addition, there is a
``plateau'' $\set{p_\eps=0}$ to smooth out the kink, see Figure
\ref{fig:proof:p_eps}.
In the following lemma we estimate the interface energy of
$\gamma_\eps$ combined with a suitable phase field $u_\eps$ based on
$p_\eps$.

\begin{figure}
  \centering
  \includegraphics{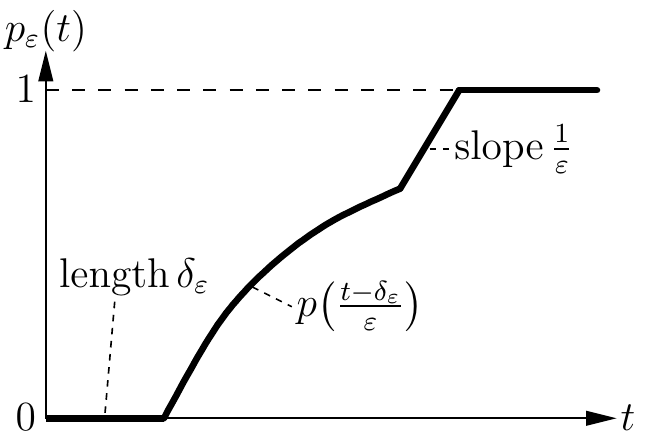}
  \caption{Construction of $p_\eps$ consisting of a ``plateau'' for
    the curve recovery, the optimal profile, and the connection to
    $1$.}
  \label{fig:proof:p_eps}
\end{figure}

\begin{lemma}
  \label{lem:proof:phase-recovery}
  Let $\gamma_\eps$ be as in Lemma \ref{lem:proof:kink-recovery}. Then
  there exists $u_\eps \in W^{1,p}(J)$ such that
  $\|u_\eps\|_{L^\infty(J)} \leq C_0$, $u_\eps = u$ on $\partial J$,
  $u_\eps \to u$ in $L^p(J)$ for any $p \in [1,\infty)$,
  $\int_{M_\eps(J)} u_\eps \,d\mu_\eps = \int_{M_\gamma(J)} u \,d\mu +
  o(\sqrt{\eps})$, and
  \begin{align*}
    \limsup_{\eps \to 0} \calI_\eps(\gamma_\eps,u_\eps,J)
    &\leq
    2\pi ( \sigma + \hat\sigma |[\gamma'](0)| ) y(0),
    \\
    \limsup_{\eps \to 0} \calH_\eps(\gamma_\eps,u_\eps,J)
    &\leq
    \calH(\gamma,u,J).
  \end{align*}
\end{lemma}

\begin{proof}
  If $s=0$ is a proper interface let
  \begin{equation*}
    u_\eps(t) = 
    \begin{cases}
      p_\eps(t)   &\text{if } 0 \leq t, \\
      -p_\eps(-t) &\text{if } 0 > t,
    \end{cases}
  \end{equation*}
  if $u(t) = \sign t$ in $J$, and the negative of it, if $u(t) = -
  \sign t$; for a ghost interface take the combination of $p_\eps(t)$
  and $p_\eps(-t)$ or its negative. Obviously, $u_\eps \to u$ in
  $L^p(J)$, $|u_\eps| \leq C_0$ in $J$, and $u_\eps = u$ on $\partial
  J$. For the energy estimates we assume $u(t) = \sign t$, the proof
  of the other cases works with the obvious changes.

  Due to $u_\eps \equiv 0$ in $J_\eps=(-\delta_\eps,\delta_\eps)$,
  Lemma \ref{lem:proof:kink-recovery} provides
  \begin{equation*}
    \limsup_{\eps \to 0}
    \calI_\eps(\gamma_\eps,u_\eps,J_\eps)
    \leq
    2\pi \hat\sigma |[\gamma'](0)| y(0).
  \end{equation*}
  Since $\gamma \in W^{2,2}(J\sm\set{0}; \bbR^2)$ and $\gamma_\eps =
  \gamma + o(1)$ in $W^{2,2}(J \sm J_\eps; \bbR^2)$, the curvature
  term in $\calI_\eps(\gamma_\eps,u_\eps, J \sm J_\eps)$ vanishes in
  the limit $\eps \to 0$. The other terms are easily estimated by
  \begin{align*}
    \frac{1}{2\pi}
    \int_{M_{\eps}(J \cap \set{t>\delta_\eps})}
    \eps |\snabla[M_{\eps}]{p_\eps}|^2
    + \frac{1}{\eps} W(p_\eps) \,d\mu_\eps
    &\leq
    \left(\sup_{[\delta_\eps,\delta_\eps+\sqrt{\eps}]} y_\eps \right)
    \int_0^{1/\sqrt{\eps}} |p'(t)|^2 + W(p(t)) \,d t
    \\
    &\quad
    +
    \|y_\eps\|_\infty \left(1-p\left(1/\sqrt{\eps}\right)\right)
    ( 1 + \sup_{[0,1]} W )
  \end{align*}
  on the positive side of $s=0$, and similarly on the other side. The
  second term on the right hand side vanishes in the limit $\eps\to0$
  because $p(1/\sqrt{\eps}) \to 1$, while in the first term the
  integral is bounded by $\sigma/2$ and the supremum converges to
  $y(0)$.
  Hence, taking the limit superior as $\eps \to 0$ proves the upper
  bound for $\calI_\eps(\gamma_\eps,u_\eps,J)$.
  The estimate for the Helfrich energy follows from
  $\calH_\eps(\gamma_\eps,u_\eps,J) = \calH_\eps(\gamma_\eps,u_\eps, J
  \sm J_\eps)$, using convergence of $u_\eps$ and $\gamma_\eps \chi_{J
    \sm J_\eps}$.
  Finally, one easily sees that
  \begin{equation*}
    \left|
      \int_{M_\gamma(J)} u \,d\mu - \int_{M_\eps(J)} u_\eps \,d\mu_\eps
    \right|
    \lesssim
    \delta_\eps +
    \int_{\delta_\eps}^{\delta_\eps+\sqrt{\eps}} (1-p_\eps) \,d t +
    \eps \left(1-p(1/\sqrt{\eps})\right),
  \end{equation*}
  and since
  \begin{equation*}
    \int_{\delta_\eps}^{\delta_\eps+\sqrt{\eps}} (1-p_\eps) \,d t
    =
    \sqrt{\eps} \int_0^1 1-p(t/\sqrt{\eps}) \,d t
    =
    o(\sqrt{\eps}),
  \end{equation*}
  the phase integral difference is also of order $\sqrt{\eps}$.
\end{proof}


\subsubsection{Axis of revolution}

Let $J_0 \subset \set{y=0}$ be an interval that is enclosed by two
intervals $J_l$, $J_r$ such that $(J_l \cup J_0 \cup J_r) \cap S =
\emptyset$,
\begin{equation*}
  \gamma'(t) =
  \begin{cases}
    (0,-1) & \text{in } J_l,
    \\
    (1,0) & \text{in } J_0,
    \\
    (0,1) &  \text{in } J_r,
  \end{cases}
\end{equation*}
and $|u| = 1$ in $J_l \cup J_r$.
The limit energy in $J_0$ is $\calF(\gamma,u,J_0) = 2\pi \hat\sigma
\calH^1(J_0)$, and this is easily recovered by setting $u_\eps = 0$
and $\gamma_\eps = \gamma + (0, 2\eps/\hat\sigma)$ in $J_0$ because
\begin{equation*}
  \calI_\eps(\gamma_\eps,u_\eps,J_0)
  =
  2\pi \int_{J_0}
  \left( \frac{1}{\eps} W(0) + \eps \kappa_{2,\eps}^2 \right) y_\eps \,d t
  =
  2\pi \int_{J_0} \frac{2}{\hat\sigma} W(0) + \frac{\hat\sigma}{2} \,d t
  =
  2\pi \hat\sigma \calH^1(J_0).
\end{equation*}
In $J_l$ and $J_r$ we use the same construction as for kinks. If for
simplicity of notation $J_r = (0,a)$, $\gamma(0) = (0,0)$, and $u
\equiv 1$ in $J_r$, we consider the approximate curve $\gamma_\eps$
given by $\gamma_\eps(0)=(0,2\eps/\hat\sigma)$ and the angle function
\begin{equation*}
  \phi_\eps(t) =
  \begin{dcases}
    \frac{\pi}{2\alpha\eps} t
    &\text{if } 0<t<\alpha\eps,
    \\
    \frac{\pi}{2}
    &\text{if } t>\alpha\eps
  \end{dcases}
\end{equation*}
together with the phase field $p_\eps$ for $\delta_\eps = \alpha\eps$.
The purpose of $\alpha = 2\pi/(\hat\sigma (\pi-2))$ is to ensure
$y_\eps(t) = t = y(t)$ for $t \geq \alpha\eps$.
Thanks to $y_\eps \geq 2\eps/\hat\sigma$ in $J_r$, we have
\begin{equation*}
  \frac{1}{2\pi}
  \int_{M_{\eps}(J_r)} \eps \kappa_{2,\eps}^2 \,d\mu_\eps
  =
  \eps \int_0^{\alpha\eps} \frac{x_\eps'^2}{y_\eps} \,d t
  \leq
  \frac{\hat\sigma}{2} \int_0^{\alpha\eps} x_\eps'^2 \,d t
  \leq
  \frac{\hat\sigma}{2} \alpha\eps,
\end{equation*}
and since the computations for all other terms of $\calI_\eps$ from
Section \ref{sec:kink:k-and-i-ub} still apply, we conclude
$\calI_\eps(\gamma_\eps,p_\eps,J_r) \to 0 = \calI(\gamma,u,J_r)$.
For the Helfrich energy we use that $\gamma_\eps$ is a vertical line
where $p_\eps \not= 0$ in $J_r$ and obtain
\begin{align*}
  \calH_\eps(\gamma_\eps,p_\eps,J_r)
  &=
  \calH_\eps((x_\eps,t),p_\eps,(\alpha\eps,a))
  =
  \int_{M_\eps(\alpha\eps,a)} p_\eps^2 H_{s}(p_\eps)^2 \,d\mu_\eps
  \\
  &\to
  \int_{M(J_r)} H_{s}(u)^2 \,d\mu
  =
  \calH(\gamma,u,J_r),
\end{align*}
where $x_\eps \equiv \int_0^{\alpha\eps} \cos p_\eps \,d t = 2 \alpha
\eps / \pi$.
The change in area when replacing $\gamma$ by $\gamma_\eps$ is of
order $\eps$, and the difference in the phase integral is of order
$o(\sqrt{\eps})$ as in Lemma \ref{lem:proof:phase-recovery}. A similar
construction applies to $u=-1$ and in $J_l$.


\subsubsection{Recovery of simple configurations}

\begin{corollary}
  Let $(\gamma,u) \in \calD \times \calQ$, $|\gamma'|=\mathrm{const}$
  in $I$, be a simple membrane as constructed in Lemma
  \ref{lem:proof:approx-noi-finite}.
  Then there exists a sequence $(\gamma_\eps,u_\eps) \in \calC
  \times \calP$ such that $(\gamma_\eps,u_\eps) \to (\gamma,u)$
  in $C^0(I; \bbR^2) \times L^1(\set{y>0})$ and $\limsup_{\eps
    \to 0} \calF_\eps(\gamma_\eps, u_\eps) \leq \calF(\gamma,u)$.
\end{corollary}

\begin{proof}
  We obtain a sequence $(\gamma_\eps,u_\eps)$ that converges in energy
  and approximates $(\gamma,u)$ in $C^0(I;\bbR^2) \times
  L^1(\set{y>0})$ by combining the local approximations for kinks,
  interfaces, and the axis of revolution with the unchanged parts of
  $(\gamma,u)$, taking into account possible $x$-shifts to join
  segments continuously.
  This sequence satisfies $\calA_{\eps} = \calA_\gamma + O(\eps) = A_0
  + O(\eps)$ and $\int_{M_{\eps}} u_{\eps} \,d\mu_{\eps} = m A_0 +
  o(\sqrt{\eps})$, and the area constraint is recovered as in Lemma
  \ref{lem:proof:approx-noc-finite}. For the phase integral let $h
  \colon J \to \bbR$ be smooth, have compact support in an interval
  $J$, where $(\gamma,u)$ is unchanged except for an $x$-shift, and
  satisfy $\int_{M_\gamma} h \,d\mu = 1$. Then $u_{\eps} +
  \alpha_{\eps} h$ satisfies the constraint if
  \begin{equation*}
    \alpha_{\eps}
    =
    \int_{M_\gamma} u \,d\mu
    -
    \int_{M_{\eps}} u_{\eps} \,d\mu_{\eps}
    =
    m A_0
    -
    \int_{M_{\eps}} u_{\eps} \,d\mu_{\eps}.
  \end{equation*}
  Convergence of $u_{\eps} + \alpha_{\eps} h \to u$ in $L^p(J)$ as
  $\eps\to0$ and of the Helfrich energy are obvious. Since
  $\alpha_{\eps}$ is of order $o(\sqrt{\eps})$, also the interface
  energy $\calI(\gamma_{\eps},u_{\eps}+\alpha_{\eps} h)$ still
  converges to $\calI(\gamma,u)$, thanks to
  \begin{equation*}
    \frac{1}{\eps} W(\pm1 + \alpha_{\eps} h)
    =
    \frac{1}{\eps} \left(
      W(\pm1) + \alpha_{\eps} h W'(\pm1) + O(\alpha_{\eps}^2)
    \right)
    =
    o(1).
    \qedhere
  \end{equation*}
\end{proof}


\section{Generalisations and open problems}
\label{sec:some-generalisations}

Finally, we discuss some extensions of Theorem
\ref{thm:model:gamma-type-conv} and open problems.
First of all, the proof is easily adapted to non-symmetric potentials
$W$. In this case one splits $\sigma$ into two constants
\begin{equation*}
  \sigma^+
  =
  \int_0^1 2 \sqrt{W(u)} \,d u
  \qquad\text{and}\qquad
  \sigma^-
  =
  \int_{-1}^0 2 \sqrt{W(u)} \,d u
\end{equation*}
and distinguishes proper interfaces and ghost interfaces in the
different phases $u=\pm1$ by the line tensions $\sigma^+ + \sigma^-$,
$2\sigma^+$, or $2\sigma^-$ instead of $\sigma$ in the limit energy.
One may also consider potentials as $W(u) = (1-u)^2$ and drop the
phase integral constraint. Then there is only one lipid phase, and
$u_\eps$ is merely an auxiliary variable that allows curvature induced
kinks in the limit.
As already stated, rigidities other than $k^\pm = -k_G^\pm = 1$ can be
considered as long as the conditions
\eqref{eq:intro:param_restrictions} hold. Also, the $u^2$ in
$\calH_\eps$ can be replaced by other continuous functions that are
equal to $0$ for $u=0$ and $1$ for $u=\pm1$.

Without change of the proof, the constraint of prescribed area for the
approximate setting can be relaxed to
\begin{equation*}
  0
  <
  \inf_{\gamma \in \calC_\eps} \calA_\gamma
  \leq
  \sup_{\gamma \in \calC_\eps} \calA_\gamma
  <
  \infty,
\end{equation*}
and thus incorporated as penalty term in the energy. The same is true
for the phase integral.
Other constraints that change continuously under the convergence
proved in Lemma \ref{lem:proof:equi-coercivity} can also be imposed,
for instance on the enclosed volume $\calV_\gamma = \pi
\int_{M_\gamma} x' y^2 \,d t$.


The arguments can be adapted to open surfaces of revolution generated
by curves $\gamma=(x,y) \colon I \to \bbR \times \bbR_{>0}$ with
prescribed boundary conditions. Alternatively, a uniform bound on $y$
derived from an energy like $\calF_\eps + \calG$, where
\begin{equation*}
  \calG(\gamma)
  =
  \int_{M_\gamma(\partial I)} d\calH^1
  =
  2\pi \sum_{s \in \partial I} y(s),
\end{equation*}
is sufficient, as it still ensures a bound on the curve length
\cite{Helmers12}; the corresponding limit $\calF+\calG$ models open
lipid bilayers with kinks, see for instance \cite{TuOu03}.
If boundary conditions for $\gamma'$ are prescribed, then kinks may
appear at the boundary in the sense that the tangent vector of the
limit curve differs from the prescribed one and contributes to the
limit energy like a ghost interface.


\subsection{Gauss curvature, axis of revolution, and full
  \texorpdfstring{$\Gamma$}{\textGamma}-limit of
  \texorpdfstring{$\calF_\eps$}{F\textepsilon}}
\label{sec:gauss-curvature-axis}

In the study of membranes it is often assumed that $k_G^+ = k_G^-$;
then the Gauss curvature integral in \eqref{eq:intro:energy} is a
topological invariant and omitted, see for instance \cite{JuLi96}.
Therefore it is desirable to drop the Gauss curvature in $\calF_\eps$
and to consider
\begin{equation*}
  \widehat \calF_\eps(\gamma,u)
  =
  \int_{M_\gamma} u^2 (H-H_{s}(u))^2 \,d\mu
  +
  \calI_\eps(\gamma,u).
\end{equation*}
Since $\widehat \calF_\eps$ still bounds the first variation of
$M_\gamma$, see Lemma \ref{lem:model:variation-bound} and the remark
after Lemma \ref{lem:surf:length-bound}, the arguments for
equi-coercivity and the lower bound in bulk and at (ghost) interfaces
still apply. At the axis of revolution we obtain the estimate
\begin{equation}
  \label{eq:aor-lb-h-only}
  \liminf_{\eps \to 0} \widehat \calF_\eps(\gamma_\eps,u_\eps,R)
  \geq
  \hat\sigma
  \liminf_{\eps \to 0}
  \int_{M_\eps(R)} |H_\eps| \,d\mu_\eps
\end{equation}
in place of \eqref{eq:proof:aor-lb}. A subsequence of the measure
$\nu_\eps = |H_\eps| \mu_\eps$ converges to some
$\nu$ in the sense of finite Radon measures, thus the right
hand side of \eqref{eq:aor-lb-h-only} is bounded from below by
$\hat\sigma\nu(R)$.
We would like to connect $\nu(R)$ or \eqref{eq:aor-lb-h-only} to the
limit curve $\gamma$, but due to the lack of good bounds on
$(\gamma_\eps)$ in $R$, we are able to do this only in special
situations. If for instance $J \subset R$ is an interval,
$|\gamma_\eps'| \equiv q_\eps$, $x_\eps'\geq0$, and $\phi_\eps$ an
angle function for $\gamma_\eps$ we can use the angle formulas
\eqref{eq:surf:curve-derivative-angle},
\eqref{eq:surf:curvatures-angle} and integrate by parts to find
\begin{equation}
  \label{eq:liminf-with-mean-curvature1}
  \frac{1}{2\pi} \int_{M_\eps(J)} H_\eps \,d\mu_\eps
  =
  q_\eps \int_J \phi_\eps \sin \phi_\eps + \cos \phi_\eps \,d t
  - \phi_\eps y_\eps|_{\partial J}
  \geq
  \calL_\eps(J)
    - \phi_\eps y_\eps|_{\partial J},
\end{equation}
where the last inequality holds due to $\phi_\eps \sin \phi_\eps +
\cos \phi_\eps \geq 1$ for $\phi_\eps \in [-\pi/2,\pi/2]$. Exhausting
$\interior R$ by such intervals $J$, we conclude
\begin{equation*}
  \liminf_{\eps \to 0}
  \int_{M_\eps(\interior R)} |H_\eps| \,d\mu_\eps
  \geq
  2\pi \liminf_{\eps \to 0}
  \calL_\eps(\interior R)
  \geq
  2\pi \calL_\gamma(\interior R).
\end{equation*}
However, $R$ is a closed set, and in general we only know
$\liminf \widehat\calF_\eps(\cdot,\cdot,R) \geq 0$. Our limit
functional is thus
\begin{equation*}
  \widehat \calF(\gamma,u)
  =
  \int_{M_\gamma} (H-H_s(u))^2 \,d\mu
  +
  2\pi \sum_{s \in S_\gamma \cup S_u}
  \left( \sigma + \hat\sigma |[\gamma'](s)| \right) y(s),
\end{equation*}
which does not provide any information about the axis of revolution at
all and, as seen before, cannot be recovered in general.

The following example satisfies $R = \partial R$ and shows that
finding the lower bound in $\interior R$ is not sufficient. Moreover,
it highlights the difference between our limit energies and the full
$\Gamma$-limit of $\calF_\eps$ and $\widehat \calF_\eps$.
Let $\rho = (x_\rho,y_\rho)$ be a periodic rectangular signal and
\begin{equation*}
  \gamma_{\eps_k}(t)
  =
  (x_k(t),y_k(t))
  =
  \begin{pmatrix}
    0 \\ \eps_k
  \end{pmatrix}
  +
  \begin{pmatrix}
    x_\rho(k t)/k^2 \\ y_\rho(k t)/k
  \end{pmatrix}
\end{equation*}
for $t$ in some interval $J$, where $\eps_k = c/k$, $0<c \ll 1$; see
Figure \ref{fig:kink:example-2}. Using the constructions from Section
\ref{sec:proof:ub}, it is straightforward to check that
$(\gamma_{\eps_k},u_{\eps_k})$ with $u_{\eps_k} \equiv 0$ can be made
admissible by smoothing the kinks, reparametrising for constant speed,
and extending $\gamma_{\eps_k}(J)$ smoothly to obtain closed surfaces
of revolution of prescribed area. Then $\widehat
\calF_{\eps_k}(\gamma_{\eps_k},u_{\eps_k})$ and
$\calF_{\eps_k}(\gamma_{\eps_k},u_{\eps_k})$ are uniformly bounded,
and in the limit membrane $(\gamma,u)$ the segment
$\gamma_{\eps_k}(J)$ collapses to a single point. Hence, if the other
segments of $\gamma$ do not touch the axis of revolution in the
interior, we have $R = \partial R$.
The length of $\gamma_{\eps_k}(J)$ is $2+1/k$, thus we find $\liminf
\calF_{\eps_k}(\gamma_{\eps_k},u_{\eps_k},J) \geq 4\pi$, but
$\calF(\gamma,u,J) = \calL_\gamma(J) = 0$.
This suggests that $\Gamma$-$\lim \calF_\eps$ is not geometric, that
is, it is not invariant under reparametrisations.

\begin{figure}
  \centering
  \includegraphics[width=.6\textwidth]{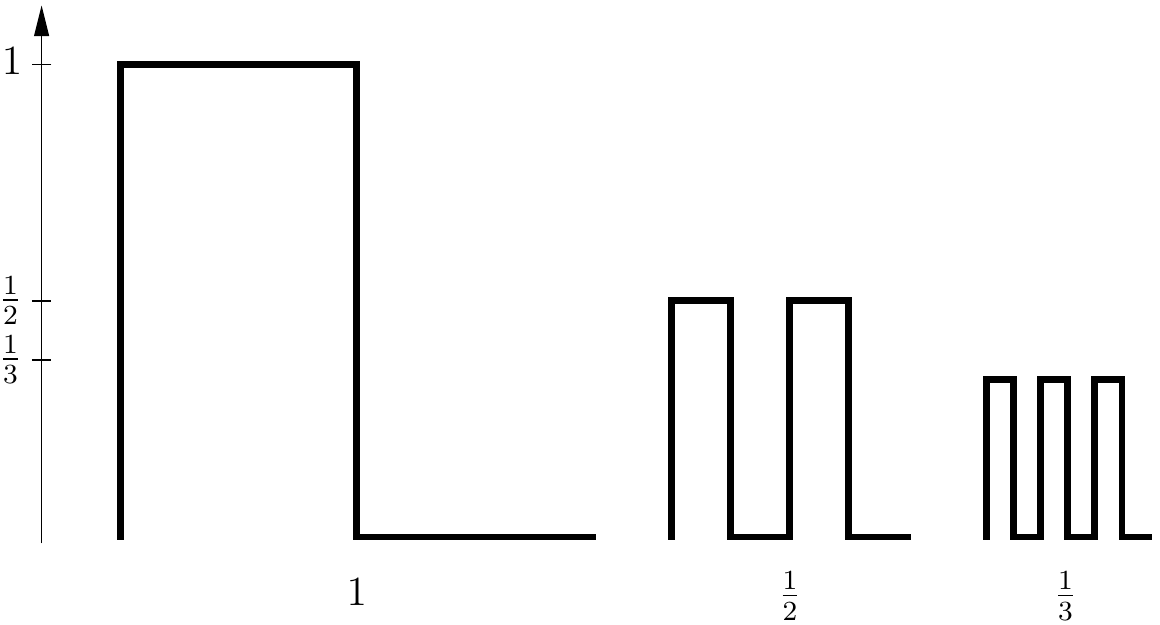}
  \caption{The shape of $\gamma_{\eps_1}$, $\gamma_{\eps_2}$,
    $\gamma_{\eps_3}$, neglecting the base height $\eps_k$.}
  \label{fig:kink:example-2}
\end{figure}

Recall, on the other hand, that the lower limit of $\calF_\eps$ at the
axis of revolution is non-negative. Since changes of $(\gamma,u) \in
\calD \times \calQ$ at the axis of revolution do not affect area or
phase integral, removing segments of $(\gamma,u)$ at the axis is
admissible and only reduces the limit energies. Minimisers of $\calF$
and $\Gamma$-$\lim \calF_\eps$ should thus have no energy at the axis
of revolution at all, and for such membranes the two energies agree.


\section*{Acknowledgements}

It is a pleasure to thank Barbara Niethammer and Michael Herrmann for
their advice and many discussions.
Gratitude is also expressed to the referees whose careful work
improved the exposition.
This work was supported by the EPSRC Science and Innovation award to
the Oxford Centre for Nonlinear PDE (EP/E035027/1).


\bibliography{paper}
\bibliographystyle{abbrv}

\end{document}